\documentclass[12pt,leqno]{article}
\usepackage[utf8]{inputenc}
\usepackage[T1]{fontenc}
\usepackage[UKenglish]{babel}
\usepackage[a4paper,margin=1in]{geometry}

\usepackage{amsmath,amsthm,amsfonts,amssymb}
\usepackage{mathtools}
\usepackage{color,enumerate}
\usepackage{mathrsfs,dsfont}
\usepackage{url,verbatim}

\theoremstyle{plain}
\newtheorem{theorem}{Theorem}
\newtheorem{lemma}[theorem]{Lemma}
\newtheorem{proposition}[theorem]{Proposition}

\theoremstyle{definition}

\newtheorem*{notation}{Notation}

\numberwithin{theorem}{section}
\numberwithin{equation}{section}

\newcommand{\Xcal}{\mathcal{X}}
\newcommand{\Lcal}{\mathcal{L}}

\newcommand{\real}{\mathds{R}}
\newcommand{\nat}{\mathds{N}}
\newcommand{\Ee}{\mathds{E}}
\newcommand{\Pp}{\mathds{P}}
\newcommand{\Ww}{\mathds{W}}
\newcommand{\I}{\mathds{1}}
\newcommand{\Normal}{\mathsf{N}}

\newcommand{\Cscr}{\mathscr{C}}
\newcommand{\Escr}{\mathscr{E}}
\newcommand{\Dscr}{\mathscr{D}}
\newcommand{\Mscr}{\mathscr{M}}
\newcommand{\Pscr}{\mathscr{P}}

\newcommand{\dd}{\mathrm{d}}
\newcommand{\eup}{\mathrm{e}}

\newcommand{\Lip}{\mathrm{Lip}}

\DeclareMathOperator{\divi}{div}

\newcommand{\normal}{\color{black}}
\usepackage{marginnote}

\begin{document}
\allowdisplaybreaks

\title{\textbf{Exact rate of convergence for the empirical measure of a subordinated process in $p$-Wasserstein distance}}

\author{\textbf{Ren\'{e} L. Schilling$^{a)}$, Bingyao Wu$^{b),a)}$}\\
\footnotesize{ $^{a)}$ Institut f\"{u}r Mathematische Stochastik, Fakult\"{a}t Mathematik, TU Dresden,  Dresden 01062, Germany}\\
\footnotesize{$^{b)}$School of Mathematics and Statistics, Fujian Normal University, Fuzhou 350007, China}\\\\
\footnotesize{rene.schilling@tu-dresden.de, bingyaowu@163.com}}

\date{}

\maketitle

\begin{abstract}\noindent
	We establish exact rates of convergence in the $p$-Wasserstein distance for the empirical measure of a class of non-symmetric  jump processes, which are subordinated to a  diffusion process on a compact Riemannian manifold. For the quadratic Wasserstein distance we determine the renormalization limit. We extend the main results of \cite{WW} and \cite{WWZ}. Our method uses two key elements: a Bernstein-type inequality for the subordinated process and the PDE approach established in \cite{AMB}.

\end{abstract}

\section{Introduction and Main Results}
It is a classical problem in probability theory to estimate the convergence rate  w.r.t.\ Wasserstein distance  at which the empirical measure of an i.i.d.\ sample approaches the reference measure. This question has its origins in the well-known AKT (Ajtai--Koml\'os--Tusn\'ady) theorem, see \cite{AKT1984}. Let us give a brief overview, see \cite{AKT1984,AMB,L17,Ta}. Suppose that $(X_i)_{i\ge 1}$ are independent, uniformly distributed, random variables $[0,1]^d$ and write $\mu$ for the uniform distribution.   The empirical measure of the sample $(X_i)_{i\geq 1}$ is defined as
\begin{gather*}
	\mu_n:=\frac 1 n \sum_{i=1}^n \delta_{X_i}.
\end{gather*}
It is known that for every $1\leq p<\infty$ the $p$-Wasserstein distance satisfies
\begin{gather*}
	\Ee[\Ww_p^p(\mu_n,\mu)]
	\sim
	\begin{cases}
		n^{-\frac p 2}, &d=1,\\
		\left(\frac{\log n}{n}\right)^{\frac p 2}, &d=2,\\
		n^{-\frac p d}, & d\geq 3.
	\end{cases}
\end{gather*}
Using a PDE approach,  Ambrosio,  Stra and Trevisan \cite{AMB}  worked out  the correct renormalization limit for the case  $p=2$ and $d=2$:
\begin{gather*}
	\lim_{n\to\infty} \frac{n}{\log n} \Ee\left[\Ww_2^2(\mu_n,\mu)\right] = \frac 1 {4\pi}.
\end{gather*}
There are important further developments in the discrete-time setting: The random variables need not be independent, and the state space need not be compact. This is, e.g., done in \cite{AG2019,AVT,Borda,HGT,GT24,L17,LW-L,WZ-A,Zhu2021} and the references given there. To the best of our knowledge, these results  do not cover all $p\in [1,\infty)$  or the renormalization.

In the continuous-time setting there exists a vast literature, for example on diffusion processes \cite{MT24+,TWZ,W21,W22,WAAP,WJEMS,WWZ,WZ,Zhu}, subordinated diffusion processes \cite{LW1,LW2,LW3,Wang23NS,WW}, fractional Brownian motion \cite{HT} and McKean--Vlasov processes \cite{DU}. The exact $p$-Wasserstein convergence rates and the form of the quadratic Wasserstein convergence limits for asymmetric diffusion processes on compact Riemannian manifolds have been established in \cite{TWZ,WWZ}. In this paper we study the exact $p$-Wasserstein convergence rates and the quadratic Wasserstein convergence limits for a class  of jump processes obtained by subordination. In order to present our main results, we first introduce the basic framework and relevant notation.

Throughout this paper, $(M,\rho)$ is a $d$-dimensional connected compact Riemannian manifold with Riemannian distance $\rho$; if $M$ has a boundary $\partial M$, we assume that it is smooth. By $\dd x$ we denote the Riemannian volume measure on $M$. For $V\in C^2(M)$ we define $\mu(\dd x) = \eup^{V(x)}\,\dd x$ and we assume that $\mu$ is a probability measure. Let $Z:=\sum_{i=1}^d Z_i \frac{\partial}{\partial x_i}$ be a $C^1$-vector field, which is divergence free, i.e.\
\begin{gather*}
	\mu(Zf)
	:= \sum_{i=1}^d\int_M Z_i(x)\frac{\partial f(x)}{\partial x_i}\,\mu(\dd x)
	=\int_M \langle Z,\nabla f\rangle \,\dd\mu
	= 0,\quad f\in C^1(M).
\end{gather*}
As usual, we identify the vector field $Z$ with the vector $(Z_1,Z_2,\dots,Z_d)$. Let $(X_t)_{t\geq 0}$ be the (reflecting, if $\partial M\neq\emptyset$) diffusion process with infinitesimal generator $L:=\Delta+\nabla V\cdot \nabla$, where $\Delta$ and $\nabla$ are the Laplace--Beltrami operator and the gradient operator on $M$, respectively. Note that due to the choice of our measure $\mu$ we have the following integration-by-parts formula:
\begin{gather}\label{eq:ibp}
	\langle (-L)u,u\rangle_{L^2(\mu)} = \|\nabla u\|^2_{L^2(\mu)}.
\end{gather}

It is well-known that both $-(-L)^\alpha$, $\alpha \in (0,1]$, and $\Lcal^\alpha:=-(-L)^\alpha+Z$ are generators of stochastic processes, see e.g.~\cite[Chapter 13]{SSV}, \cite[Section 1.15.5, Section C.5]{BLG}. It is not difficult to check that the $L^2(\mu)$-adjoint $(\Lcal^\alpha)^*$ of $\Lcal^\alpha$ in $L^2(\mu)$ is of the form $-(-L)^\alpha-Z$. Therefore, the process $(\Xcal_t^\alpha)_{t\geq 0}$  generated by $\Lcal^\alpha$ is not symmetric in $L^2(\mu)$ unless $Z= 0$. Let $(P_t^\alpha)_{t\geq 0}$ and $(\widehat{P}_t^\alpha)_{t\geq 0}$ be the  semigroup generated by $\Lcal^\alpha$ and $-(-L)^\alpha$ respectively. The empirical measure of $(\Xcal_t^\alpha)_{t\geq 0}$ is given by
\begin{gather*}
	\mu_T^\alpha(\cdot)
	:= \frac 1 T \int_0^T \delta_{\Xcal_t^\alpha}(\cdot)\,\dd t,\quad T>0,
\end{gather*}
where $\delta_{y}(\cdot)$ is the Dirac measure at the point $y$. It is shown in \cite{Wu2}, that for any $f\in L^2(\mu)$, $f\neq 0$, with $\mu(f):=\int_Mf\,\dd\mu=0$, we have
\begin{gather*}
	\sqrt{T} \mu_T^\alpha(f) \xrightarrow{T\to\infty} \Normal(0, 2 \Sigma(f))\quad \text{weakly},
\end{gather*}
where $\Normal(0, 2 \Sigma(f))$ denotes the normal distribution with mean $0$ and variance $2\Sigma(f)$. The variance $\Sigma(f)$ is defined as
\begin{gather}\label{Sigma}
	\Sigma(f)
	:= \int_0^\infty \mu(fP_t^\alpha f)\,\dd t\in (0,\infty).
\end{gather}

By $\Pscr$ we denote the set of probability measure on $M$. The $p$-Wasserstein distance of $\mu,\nu\in\Pscr$ is defined as
\begin{gather*}
	\Ww_p^{p\vee 1}(\mu,\nu)
	:= \inf_{\pi\in\Cscr(\mu,\nu)}\left(\int_{M\times M}\rho(x,y)^p\,\pi(\dd x,\dd y)\right)^{\frac 1{p\vee 1}},
	\quad
	p>0,
\end{gather*}
where $\Cscr(\mu,\nu)$ is the set of all couplings of $\mu$ and $\nu$, i.e.\ $\pi$ is a probability measure on $M\times M$ such that $\mu$ and $\nu$ are its marginals.

By the Birkhoff ergodic theorem, see e.g.\ \cite[Theorem 9.6 and Theorem 9.8]{K} or \cite[Theorem 1.1]{S}, the empirical measure $\mu_T^\alpha$ converges weakly to  the  invariant measure $\mu$ as $T\to\infty$. As shown in \cite[Corollary 6.13]{V}, this is equivalent to, for any $p\ge 1$, $\Ww_p(\mu_T,\mu)\to 0$ as $T\to\infty$. It is, therefore, natural to study the rate of convergence of $\mu_T^\alpha$.

Since $M$ is compact and connected, the eigenvalues of $-L$ (with Neumann boundary conditions if $\partial M\neq\emptyset$) are discrete and can be listed (counting multiplicities) in increasing order as:
\begin{gather*}
	0 < \lambda_1 \leq \lambda_2 \leq \dots \leq \lambda_i\le\dots
\end{gather*}
By $\{\phi_i\}_{i\in\nat}\subset L^2(\mu)$ we denote the corresponding eigenfunctions of $-L$; moreover, we assume that $\|\phi_i\|_{L^2(\mu)}=1$. The eigenvalues satisfy the following quantitative estimate
\begin{gather}\label{lam}
	c_1 i^{\frac 2 d} \leq \lambda_i \leq c_2 i^{\frac 2 d},
	\quad i\in \nat,
\end{gather}
where $c_1,c_2$ are constants, which do not depend on $i$. The Markov semigroup $P_t$ generated by $L$ has a symmetric heat kernel $p_t$ with respect to $\mu$,  which is given by
\begin{gather}\label{HS2}
	p_t(x,y)
	= 1 + \sum_{i=1}^\infty \eup^{-\lambda_i t} \phi_i(x)\phi_i(y),
	\quad t>0, \; x,y\in M,
\end{gather}
We refer to \cite[Sections 1.7.2, A.4]{BLG} and \cite[Exercise III.38, Note III.15]{Chavel} as standard references.  We need a few properties of the semigroup $(P_t)_{t\geq 0}$, which can be found in \cite{BLG,W14,WZ,Zhu}. Since $M$ is a $d$-dimensional compact manifold, there exists some constant $c>0$ such that
\begin{gather*}
	\|P_t-\mu\|_{1\to\infty}
	\leq c(1\wedge t)^{-\frac d 2},
	\quad t>0.
\end{gather*}
Combining this with the contractivity of $(P_t)_{t\geq 0}$ on $L^p(\mu)$ for any $1\leq p\leq \infty$, and an interpolation argument \cite[Theorem 0.3.13]{W05}, we see that there exists some constant $c>0$ such that for any $1\leq p\leq q\le\infty$,
\begin{gather}\label{ULT}
	\|P_t-\mu\|_{p\to q}\leq c (1\wedge t)^{-\frac{d(q-p)}{2pq}}\eup^{-\lambda_1 t},\quad t>0,
\end{gather}
see e.g. \cite[(2.12)]{WJEMS}.

We can now present our main results on the Wasserstein convergence of $\mu_T^\alpha$. For any $(\alpha,d)\in(0,1]\times \nat$, and large enough $T$, define
\begin{gather*}
	\gamma_{\alpha,d}(T)
	:=
	\begin{cases}
	T^{-\frac 1 2},\quad &d<2(1+\alpha),\\
	T^{-\frac 1 2}\log^{\frac 1 2}T,\quad &d=2(1+\alpha),\\
	T^{-\frac 1 {d-2\alpha}},\quad &d>2(1+\alpha),
	\end{cases}
\end{gather*}

\begin{theorem}[Upper bound]\label{Wpq}
	Let $(\alpha,T,p,q)\in(0,1]\times (0,\infty)\times [1,\infty)\times (0,\infty)$ and denote by $q^*$ the conjugate exponent of $\frac pq\vee 1$. If $\nu = h_\nu \mu\in\Pscr$ for some where $h_\nu\in L^{q^*}(\mu)$, then
	\begin{gather*}
	\Ee^\nu[\Ww_p^q(\mu_{T}^\alpha,\mu)]
	\lesssim \|h_\nu\|_{L^{q^*}(\mu)} \gamma_{\alpha,d}(T)^q,\quad T \gg 1.
	\end{gather*}
\end{theorem}

\begin{theorem}[Lower bound]\label{Wpq-L}
	Let $(\alpha,T,p,q)\in(1/2,1]\times (0,\infty)\times [1,\infty)\times (0,\infty)$. If $Z\neq 0$, then
	\begin{gather*}
	\inf_{\nu\in\Pscr}\Ee^\nu[\Ww_p^q(\mu_{T}^\alpha,\mu)]
	\gtrsim\gamma_{\alpha,d}(T)^q, \quad T \gg 1.
	\end{gather*}
	If $Z=0$, this estimate still holds for $\alpha\in (0,1]$.
\end{theorem}

\begin{theorem}\label{Limit-R}
	For $(\alpha,d)=(1/2,3)$ and $Z=0$ one has
	\begin{gather*}
		\sup_{x\in M}\left|\frac T {\log T}\Ee^x[\Ww_2^2(\mu_T^\alpha,\mu)]-\frac{\mathrm{Vol}(M)}{2\pi^2}\right|
		\lesssim  \sqrt{\frac{\log\log T}{\log T}},\quad T\gg 1.
	\end{gather*}
\end{theorem}

\begin{notation}
	Most of our notation is standard or self-explanatory. As usual, $\|\cdot\|_{p\to q}$ denotes the operator norm from $L^p(\mu)$ to $L^q(\mu)$, $1\leq p\leq q\le\infty$. The shorthand $T\gg 1$ means that $T$ is much greater than $1$. We use the shorthand $A\lesssim B$ to indicate that there exists a constant $c>0$ such that $A\leq cB$. The constant $c$ is, in general, independent of $T$, but it may (and will) depend on the underlying manifold and the parameters $\alpha$ and $d$. The notation $\Ee^\nu (\dots) = \int_{M}\Ee ^x(\dots)\,\nu(\dd x)$ stands for the expectation if the process has the initial distribution $\nu\in\Pscr$. For any $a,b\in\real$, we define $a\vee b:=\max\{a,b\}$ and $a\wedge b:=\min\{a,b\}$. $\nat$ is the set of all positive integers. We use $C_{b}^{\Lip}$ for the set of bounded, Lipschitz continuous functions.
\end{notation}

\section{A deviation inequality}
 The following deviation inequality will be central for the proof of Theorems~\ref{Wpq}--\ref{Limit-R}.
Let $\alpha\in(0,1]$ and
\begin{gather*}
	\Escr^{\alpha}(f,g)
	:= \frac 1 2 \left[ \left\langle - \Lcal^\alpha f,g\right\rangle_{L^2(\mu)} + \left\langle - \Lcal^\alpha g,f\right\rangle_{L^2(\mu)} \right],
	\quad f,g \in\Dscr(- \Lcal^\alpha)
\end{gather*}
be the symmetric Dirichlet form generated by the  (symmetric part of the)  operator $\Lcal^\alpha$. The basic idea borrows from the case when $\alpha=1$ in \cite{WWZ}, and now we extend it to more general cases. From \cite[Theorem 1]{wu2000deviation}, we know the following deviation inequality  for the (non-symmetric) process $(\Xcal_t^\alpha)_{t\geq 0}$ generated by $\Lcal^\alpha$.
\begin{lemma} \label{LDT}
	Let $\xi\in (0,\infty)$ and $g \in L^1(\mu)$ such that $\mu(g) = 0$. Define
	\begin{gather*}
		\mathcal{I}_g^\alpha(\xi-)
		:= \lim_{\epsilon \to 0}\inf \left\{ \Escr^{\alpha}(h,h) \::\:  \, h \in  \Dscr(\Escr^{\alpha}), \; \mu(h^2) = 1,  \; \mu(|g| h^2 ) < \infty,\; |\mu(g h^2)| = \xi-\epsilon \right\},
	\end{gather*}
	\textup{(}$\inf \emptyset := \infty$\textup{)}. Then,  for all probability measures $\nu = h_\nu\mu$, one has
	\begin{gather*}
		\Pp^\nu \left( \left| \frac 1 T \int_{0}^{T} g(\Xcal_t^\alpha) \,\dd t \right| > \xi \right)
		\leq 2  \left\| h_\nu \right\|_{L^2(\mu)} \exp \left[- T \mathcal{I}_g^\alpha(\xi-) \right],
		\quad T,\xi>0.
	\end{gather*}
\end{lemma}

We will now give a more manageable form for the exponential appearing in Lemma~\ref{LDT}.  For any $g\in \mathscr{D}((-L)^{-\frac \alpha 2})$  with $\mu(g)=0$, we define the following two parameters:
\begin{gather}\begin{aligned}\label{sm}
	 \sigma^2(g)
	 &:=2\|(-L)^{-\frac \alpha 2} g\|_{L^2(\mu)}^2,\\[9pt]
	 m(g)
	 &:=
	 \begin{cases}
	 	\|g\|_{L^1(\mu)}, & d<2\alpha,\\
		 \inf_{1/\alpha<p<\infty} \left[\left(\frac p {\alpha p-1}\right)^2\|g\|_{L^p(\mu)}\right], & d=2\alpha,\\
	 	\|g\|_{L^\frac d {2\alpha} (\mu)},& d>2\alpha.
	 \end{cases}
\end{aligned}\end{gather}
Notice that $m(g)\in [0,\infty]$ is  (formally)  defined for all measurable functions $g$, but we will see below that only $m(g)<\infty$ leads to non-trivial results.

\begin{proposition}\label{bern}
	There exists a constant $\gamma>0$ such that for any  $g\in \mathscr{D}((-L)^{-\frac \alpha 2})$  with $\mu(g)=0$ and any $\nu=h_\nu\mu\in\Pscr$,
	\begin{gather*}
	\Pp^\nu\left(\left|\frac 1 T\int_0^T g(\Xcal_t^\alpha)\,\dd t\right|>\xi\right)
	\leq 2\|h_\nu\|_{L^2(\mu)}\exp\left[-\frac{T\xi^2}{2\sigma^2(g)+\gamma m(g)\xi}\right],\quad T,\xi>0.
	\end{gather*}
\end{proposition}

\begin{proof}
Since the bounded Lipschitz-continuous functions  $C_{b}^{\Lip}(M)$ are  dense in $\Dscr(\Escr^\alpha)$, we may assume that $h\in C_{b}^{\Lip} (M)$. In view of Lemma \ref{LDT}, it is sufficient to prove the following claim: {\itshape there exists a constant $\gamma>0$ such that for any $h\in C_{b}^{\Lip}(M)$ with $\mu(h^2)=1$,
\begin{gather}\label{I0}
	\frac{|\mu(g h^2)|^2}{2\sigma^2(g)+\gamma m(g)|\mu(gh^2)|}
	\leq \Escr^\alpha(h,h).
\end{gather}}
The inequality \eqref{I0} follows from
\begin{gather}\label{I1}
	|\mu(gh^2)|
	\leq \frac\gamma2 m(g)\Escr^\alpha(h,h) + \sqrt{2\sigma^2(g)\Escr^\alpha(h,h)},
\end{gather}
since we have
\begin{gather*}
	\left(|\mu(gh^2)| - \frac\gamma2 m(g)\Escr^\alpha(h,h)\right)^2
	\leq
	2\sigma^2(g)\Escr^\alpha(h,h) 	
	\leq
	2\sigma^2(g)\Escr^\alpha(h,h) +
	\left(\frac \gamma2 m(g)\Escr^\alpha(h,h)\right)^2.
\end{gather*}
From this it is easy to get \eqref{I0}.

We will now turn to the proof of \eqref{I1}. Set $\hat{h}:=h-\mu(h)$. Since $\mu(g)=0$, we see
\begin{gather}\label{eq-mugh}
	\mu(g h^2) = 2\mu(h)\mu(g\hat{h})+\mu(g\hat{h}^2).
\end{gather}
Since $\divi_{\mu}(Z)=0$, we have for any $h \in C_{b}^{\Lip}(M)$
\begin{gather*}
	\Escr^{\alpha}(h,h)
	= \left\langle - \Lcal^\alpha h,h\right\rangle_{L^2(\mu)}
	= \left\langle (-L)^\alpha h-Z h,h\right\rangle_{L^2(\mu)}
	= \left\langle (-L)^\alpha h,h\right\rangle_{L^2(\mu)}.
\end{gather*}
Because of $(-L)^{\alpha/2} 1=0$ we see that
\begin{gather*}
	\Escr^\alpha(\hat h,\hat h)
	= \Escr^\alpha(h,h).
\end{gather*}
These two identities will be frequently used when we estimate the two terms on the right-hand side of \eqref{eq-mugh}.

\medskip\noindent
\textbf{Step 1:} Estimate of the term $2\mu(h)\mu(g\hat{h})$ in \eqref{eq-mugh}. Due to the symmetry of $(-L)^{\alpha/2}$ in $L^2(\mu)$ and the Cauchy-Schwarz inequality, we have
\begin{align*}
	\left|\mu(g\hat{h})\right|
	= \left|\left\langle (-L)^{\alpha/2}((-L)^{-\alpha/2} g),\hat{h}\right\rangle_{L^2(\mu)}\right|
	&= \left|\left\langle (-L)^{-\alpha/2} g,(-L)^{\alpha/2}\hat{h}\right\rangle_{L^2(\mu)}\right|\\
	&\leq \sqrt{\frac{\sigma^2(g)\Escr^\alpha(h,h)}{2}}.
\end{align*}
Since $|\mu(h)|\leq \|h\|_{L^2(\mu)} = 1$, we see
\begin{gather}\label{I2}
	|2\mu(h)\mu(g\hat{h})|\le\sqrt{2\sigma^2(g)\Escr^\alpha(h,h)}.
\end{gather}

\smallskip\noindent
\textbf{Step 2:} Estimate of the term $\mu(g\hat{h}^2)$ in \eqref{eq-mugh}. We need to show that $\mu(g\hat{h}^2)\leq \frac{\gamma} 2 m(g)\Escr^\alpha(h,h)$. We do this separately for $d\in [1,2\alpha)$, $d=2\alpha$ and $d>2\alpha$.

\smallskip\noindent
\textbf{Case 1:} Let $d\in[1,2\alpha)$. Recall that
\begin{gather*}
	(-L)^{-\alpha}
	= \frac 1 {\Gamma(\alpha)} \int_0^\infty t^{\alpha-1}P_t\,\dd t,\quad \alpha>0,
\end{gather*}
where $(P_t)_{t\geq 0}$ is the semigroup generated by $L$. From  the fact that $(-L)^{\frac\alpha 2}1=0$, hence $\mu((-L)^{\frac\alpha 2}\hat h)=0$ and  \eqref{ULT} we obtain
\begin{align*}
	\|\hat{h}\|_{L^\infty(\mu)}
	= \|(-L)^{-\frac{\alpha} 2}(-L)^{\frac{\alpha}{2}}\hat{h}\|_{L^\infty(\mu)}
	&=\frac 1 {\Gamma(\frac {\alpha} 2)}\left\|\int_0^\infty t^{\frac {\alpha} 2-1}P_t(-L)^{\frac \alpha 2}\hat h \,\dd t\right\|_{L^\infty(\mu)}\\
	&=\frac 1 {\Gamma(\frac {\alpha} 2)}\left\|\int_0^\infty t^{\frac {\alpha} 2-1}\left(P_t-\mu\right)(-L)^{\frac \alpha 2}\hat h \,\dd t\right\|_{L^\infty(\mu)}\\
	&\leq c_1\int_0^\infty t^{\frac \alpha 2-1}(1\wedge t)^{-\frac d 4}\eup^{-\lambda_1 t}\,\dd t\, \|(-L)^{\frac \alpha 2}\hat{h}\|_{L^2(\mu)}\\
	&=c_2\sqrt{\Escr^\alpha(h,h)}
\end{align*}
for some constants $c_1,c_2>0$, the last equality is due to $d<2\alpha$. From this we obtain
\begin{gather*}
	|\mu(g\hat{h}^2)|\leq c_2^2\|g\|_{L^1(\mu)}\Escr^\alpha(h,h)  = c_2^2 m(g)\Escr^\alpha(h,h),
\end{gather*}
and we are done by taking $\gamma=2c_2^2$.

\medskip\noindent
\textbf{Case 2:} Let $d=2\alpha$. By the H\"{o}lder inequality, we get for any $p>1/\alpha$,
\begin{gather}\label{d2a}
	|\mu(g\hat{h}^2)|
	\leq \|g\|_{L^p(\mu)}\|\hat{h}^2\|_{L^{\frac p {p-1}}(\mu)}
	= \|g\|_{L^p(\mu)}\|\hat{h}\|_{L^{\frac {2p} {p-1}}(\mu)}^2.
\end{gather}
Using a similar argument as in the first case, we can use \eqref{ULT} to get some constants $c_3,c_4>0$ such that
\begin{align*}
	\|\hat{h}\|_{L^{\frac {2p} {p-1}}(\mu)}
	=\|(-L)^{-\frac \alpha 2}(-L)^{\frac \alpha 2}\hat{h}\|_{L^{\frac {2p} {p-1}}(\mu)}
	&=\frac 1 {\Gamma(\frac{\alpha} 2)}\left\|\int_0^\infty t^{\frac \alpha 2-1}P_t(-L)^{\frac \alpha 2}\hat{h}\,\dd t \right\|_{L^{\frac{2p}{p-1}}(\mu)}\\
	&\leq \frac 1 {\Gamma(\frac{\alpha} 2)}\int_0^\infty t^{\frac \alpha 2-1} \|P_t(-L)^{\frac \alpha 2}\hat{h}\|_{L^{\frac{2p}{p-1}}(\mu)}\,\dd t\\
	&\leq \frac 1 {\Gamma(\frac{\alpha} 2)}\int_0^\infty t^{\frac \alpha 2-1}\left\|P_t-\mu\right\|_{ 2\to \frac{2p}{p-1}} \|(-L)^{\frac \alpha 2}\hat{h}\|_{L^2(\mu)}\,\dd t\\
	&\leq c_3 \int_0^\infty t^{\frac \alpha 2-1}(1\wedge t)^{-\frac 1 {2p}}\eup^{-\lambda_1 t}\,\dd t\, \sqrt{\Escr^\alpha(h,h)}.
\end{align*}
We can bound the integral by
\begin{gather*}
	\int_0^1 t^{\frac{\alpha} 2 -\frac 1 {2p}-1}\,\dd t + \int_1^\infty t^{\frac{\alpha}2-1}\eup^{-\lambda_1 t}\,\dd t
	\leq\frac {c_4 p}  {\alpha p-1}+c_4
	\leq 2c_4 \frac p  {\alpha p-1},
\end{gather*}
where we use that  $1\leq \frac p {\alpha p-1}$. We insert this into \eqref{d2a}, take $\gamma=2 c_4^2$ and $m(g)=\inf_{1/\alpha<p<\infty}(\frac{p}{\alpha p-1})^2\|g\|_{L^p(\mu)}$, to get $\mu(g\hat h^2)\leq \frac\gamma2 m(g)\Escr^\alpha(h,h)$.

\medskip\noindent
\textbf{Case 3:} Let $d>2\alpha$. By the H\"{o}lder inequality, we have
\begin{gather}\label{dg2a}
	|\mu(g\hat{h}^2)|
	\leq \|g\|_{L^{\frac d {2\alpha}}(\mu)}\|\hat{h}\|_{L^{\frac {2d} {d-2\alpha}}(\mu)}^2.
\end{gather}
We first recall the following  Sobolev-type  inequality, which is from \cite[Chapter 3]{Hebey} or \cite[Proof Theorem 2.3 (3)]{Wang23NS}.
For any $s_2\geq s_1>-\infty$ and $p_1>p_2\geq 1$ satisfying
\begin{gather*}
	s_1 - \frac d{p_1} =s_2 - \frac d {p_2},
\end{gather*}
there exists some a constant $c_5>0$ such that
\begin{gather}\label{S}
	\|(-L)^{\frac {s_1} 2}f\|_{L^{p_1}(\mu)}
	\leq c_5\|(-L)^{\frac {s_2} 2} f\|_{L^{p_2}(\mu)},\quad \mu(f)=0.
\end{gather}
If we use
\begin{gather*}
	s_1=0,\quad s_2=\alpha,\quad p_1=\frac{2d}{d-2\alpha},\quad p_2=2,
\end{gather*}
in \eqref{S}, then we have
\begin{gather*}
	\|\hat{h}\|_{L^{\frac {2d} {d-2\alpha}}(\mu)}^2
	\leq c_6\|(-L)^{\frac \alpha 2}\hat{h}\|_{L^2(\mu)}^2
	= c_6\Escr^\alpha(h,h)
\end{gather*}
for some constant $c_6>0$. Taking $\gamma=2c_6$, we  get $\mu(g\hat h^2) \leq \frac\gamma2 \Escr^\alpha(h,h)$.

This finishes Step~2, and the proof is complete.
\end{proof}

\section{The upper bound -- proof of Theorem~\ref{Wpq}}

The following general upper bound for the Wasserstein distance will be an important ingredient for our proof. This bound is derived via the Benamou-Brenier formula, see, for example \cite{L17, AMB, W22}.
\begin{lemma}\label{Wp-U}
	Let $f\in L^p(\mu)$, $p\geq 2$, be a probability density function with respect to $\mu$. Then
	\begin{gather*}
		\Ww_p^p(f\mu,\mu)
		\leq p^p\int_M \left|\nabla(-L)^{- 1 }(f-1)\right|^p\,\dd \mu.
	\end{gather*}
	Moreover, if $p=2$, it also holds
	\begin{gather*}
		\Ww_2^2(f\mu,\mu)\leq \int_M\frac{|\nabla(-L)^{- 1 }(f-1)|^2}{\Mscr(f)}\,\dd\mu,
	\end{gather*}
	where
	$\Mscr(a):=\frac{a-1}{\log a}$ for $a>0$ and  $\Mscr(a):=0$ if $a=0$ or $a=1$.
\end{lemma}

We will now show the key lemma for the proof of Theorem~\ref{Wpq}.
\begin{lemma}\label{pp}
	Assume that $\alpha\in(0,1]$. For any  $p\geq 1$ one has
	\begin{gather*}
		\Ee^\mu\left[\Ww_p^p(\mu_{T}^\alpha,\mu)\right]
		\lesssim \gamma_{\alpha,d}(T)^p,\quad T\gg 1.
	\end{gather*}
\end{lemma}
\begin{proof}
We define the modified  empirical  measure as $\mu_{T,\epsilon}^\alpha(g) :=\mu_{T}^{\alpha} (P_\epsilon g)$.  By construction, $\mu_{T}^{\alpha} P_\epsilon$ is absolutely continuous with respect to $\mu$, and its Radon-Nikodym derivative is given by
\begin{gather}\label{eq:ftalpha}
	f_{T,\epsilon}^\alpha(y)
	:= \frac 1 T \int_0^T p_\epsilon( \Xcal_t^\alpha,y)\, \dd t,\quad y\in M.
\end{gather}
From the proof of \cite[Lemma 5.1(2)]{Wang23NS} we know that for any $\epsilon>0$,
\begin{gather*}
	\Ee^\mu\left[\rho(X_0,X_{\epsilon})^p\right]
	\lesssim \epsilon^{\frac p 2}.
\end{gather*}
This, together with the basic coupling $\mu_T^\alpha\otimes\mu_{T,\epsilon}^\alpha$ in  $\Cscr(\mu_T^\alpha,\mu_{T,\epsilon}^\alpha)$, gives
\begin{gather*}
	\Ee^\mu \left[\Ww_p^p(\mu_T^\alpha,\mu_{T,\epsilon}^\alpha)\right]
	\leq \frac 1 T\int_0^T\Ee^\mu\left[\rho(\Xcal_t^\alpha,\Xcal_{t+\epsilon}^\alpha)^p\right] \dd t
	= \frac 1 T\int_0^T\Ee^\mu\left[\rho(X_0,X_{\epsilon})^p\right] \dd t
	\lesssim \epsilon^{\frac p 2},
\end{gather*}
where the equality is due to the fact that the processes $(\Xcal_t^\alpha)_{t\geq 0}$ and $(X_t)_{t\geq 0}$ share the same invariant measure. Thus, by the triangle inequality of $\Ww_p$, we have
\begin{gather}\label{W-Tri}
	\Ee^\mu \left[\Ww_p^p(\mu_T^\alpha,\mu)\right]
	\lesssim \Ee^\mu \left[\Ww_p^p(\mu_T^\alpha,\mu_{T,\epsilon}^\alpha)\right]
	+ \Ee^\mu\left[\Ww_p^p(\mu_{T,\epsilon}^\alpha,\mu)\right]
	\lesssim \epsilon^{\frac p 2} + \Ee^\mu\left[\Ww_p^p(\mu_{T,\epsilon}^\alpha,\mu)\right].
\end{gather}
The second term on the right in the above inequality is also called the \emph{modified term}.

\medskip\noindent
\textbf{Estimate of the modified Term:} We apply Lemma \ref{Wp-U} with  $f=f_{T,\epsilon}^\alpha$ and use the  layer-cake formula to get
\begin{align*}
	\Ee^\mu\left[\Ww_p^p(\mu_{T,\epsilon}^\alpha,\mu)\right]
	&\lesssim\Ee^\mu \left[\|(-L)^{-\frac 1 2}(f_{T,\epsilon}^\alpha-1)\|_{L^p(\mu)}^p\right]\\
	&= p\int_M \int_0^\infty \Pp^\mu\left[\left|\frac 1 T \int_0^T(-L)^{-\frac 1 2}_{ y}\left(p_\epsilon( \Xcal_t^\alpha, y)-1\right) \dd t\right|>\xi\right]\xi^{p-1}\,\dd\xi\,\mu(\dd y).
\end{align*}
We use the shorthand $g := g(y) := p_\epsilon(X_t^\alpha,y)-1$. From Proposition~\ref{bern} we see
\begin{align*}
	&\Ee^\mu\left[\Ww_p^p(\mu_{T,\epsilon}^\alpha,\mu)\right]
	\lesssim \int_{M}\int_0^\infty \exp\left(-\frac{T\xi^2}{2\sigma^2(g)+\gamma m(g)\xi}\right)\xi^{p-1}\,\dd\xi\,\dd \mu\\
	&\quad\lesssim\int_M\int_{0}^{\frac{2\sigma^2(g)}{\gamma m(g)}}\exp\left(-c_1\frac{T\xi^2}{\sigma^2(g)}\right)\xi^{p-1}\,\dd\xi\,\dd\mu
	+\int_{M}\int_{\frac{2\sigma^2(g)}{\gamma m(g)}}^\infty\exp\left(-c_1\frac{T\xi}{m(g)}\right)\xi^{p-1}\,\dd \xi\,\dd\mu\\
	&\quad\lesssim \int_M\int_0^\infty \exp(-c_2 u)u^{\frac p 2-1}\sigma(g)^p T^{-\frac p 2}\,\dd u\,\dd\mu
	+\int_M\int_0^\infty \exp(-c_2 u) u^{p-1} m(g)^p T^{-p}\,\dd u\,\dd\mu.
\end{align*}
This means that we have
\begin{gather*}
	\Ee^\mu\left[\Ww_p^p(\mu_{T,\epsilon}^\alpha,\mu)\right]
	\lesssim T^{-\frac p 2}A_\epsilon + T^{-p}B_\epsilon,\quad T>0,
\end{gather*}
where
\begin{align*}
	A_\epsilon &:= \int_M\sigma\left((-L)^{-\frac 1 2}(p_\epsilon(\cdot,y)-1)\right)^p\,\mu(\dd y),\\
	B_\epsilon &:= \int_M m\left((-L)^{-\frac 1 2}(p_\epsilon(\cdot,y)-1)\right)^p\,\mu(\dd y).
\end{align*}
We are going to bound these terms separately.

\medskip\noindent
\textbf{Estimate of $A_\epsilon$:}  Using the spectral representation \eqref{HS2} of the heat kernel $p_\epsilon(x,y)$ along with $(-L)^{-\frac{1+\alpha}2}\phi_i = \lambda_i^{-\frac{1+\alpha}{2}}\phi_i$ and $\langle \phi_i,\phi_j\rangle_{L^2(\mu)}=\delta_{ij}$, yields
\begin{align*}
\sigma\left((-L)^{-\frac 1 2}(p_\epsilon(\cdot,y)-1)\right)^2
&= 2\left\|(-L)^{-\frac12 (1+\alpha)}(p_\epsilon(\cdot,y)-1)\right\|_{L^2(\mu)}^2\\
&= 2\int_M\left|(- L)^{-\frac{1+\alpha} 2}\sum_{i=1}^\infty\eup^{-\lambda_i\epsilon}\phi_i(x)\phi_i(y)\right|^2\,\mu(\dd x)\\
&= 2\sum_{i=1}^\infty\frac{\eup^{-2\lambda_i\epsilon}}{\lambda_i^{1+\alpha}}\phi_i^2(y).
\end{align*}
Since
\begin{gather*}
	\lambda_i^{-(1+\alpha)}
	=\frac 1 {\Gamma(1+\alpha)}\int_0^\infty \eup^{-\lambda_i s}s^\alpha \,\dd s,\quad \alpha>0,\;i\geq 1,
\end{gather*}
and
\begin{gather*}
	\sum_{i=1}^\infty \eup^{-\lambda_i \epsilon}\phi_i(y)^2
	= p_\epsilon(y,y)-1,\quad \epsilon>0,\;y\in M,
\end{gather*}
we get for any $\epsilon\in(0,1)$ and $y\in M$,
\begin{gather}\label{SUM-E}\begin{aligned}
	\sum_{i=1}^\infty\frac{\eup^{-2\lambda_i\epsilon}}{\lambda_i^{1+\alpha}}\phi_i^2(y)
	&= \frac 1 {\Gamma(1+\alpha)}\sum_{i=1}^\infty\int_0^\infty \eup^{-\lambda_i(s+2\epsilon)}s^{\alpha}\phi_i^2(y)\,\dd s\\
	&= \frac 1 {\Gamma(1+\alpha)}\int_0^\infty s^{\alpha}(p_{s+2\epsilon}(y,y)-1)\,\dd s.
\end{aligned}\end{gather}
By \eqref{ULT}, we have
\begin{gather}\label{Heat}
	\sup_{y\in M}|p_{s+2\epsilon}(y,y)-1|
	\leq \|P_{s+2\epsilon}-\mu\|_{ 1\to\infty}
	\lesssim (s+2\epsilon)^{-\frac d 2}\eup^{-\lambda_1 s},\quad \epsilon,\;s>0.
\end{gather}
Inserting this into \eqref{SUM-E}, we arrive at
\begin{gather}\label{sHT}\begin{aligned}
	\sup_{y\in M}\sum_{i=1}^\infty\frac{\eup^{-2\lambda_i\epsilon}}{\lambda_i^{1+\alpha}}\phi_i^2(y)
	&\lesssim\int_0^\infty s^{\alpha}(s+2\epsilon)^{-\frac d 2}\eup^{-\lambda_1 s}\,\dd s\\
	&\lesssim 1+\int_0^1 s^\alpha(s+2\epsilon)^{-\frac d 2}\,\dd s\\
	&\lesssim
	\begin{cases}
	1, 	&d<2(1+\alpha),\\
	\log(1+\epsilon^{-1}), & d=2(1+\alpha),\\
	\epsilon^{-\frac 12[d-2(1+\alpha)]}, & d>2(1+\alpha).
	\end{cases}
\end{aligned}\end{gather}
From the definition of $A_\epsilon$ we see for any $\epsilon>0$ that
\begin{gather}\label{A}\begin{aligned}
	A_\epsilon
	\lesssim
	\begin{cases}
	1, &d<2(1+\alpha),\\
	\log^{\frac p 2}(1+\epsilon^{-1}), & d=2(1+\alpha),\\
	\epsilon^{-\frac14 [d-2(1+\alpha)]}, & d>2(1+\alpha).
	\end{cases}
\end{aligned}\end{gather}

\smallskip\noindent
\textbf{Estimate of $B_\epsilon$:} From \eqref{Heat} it follows that
\begin{gather}\label{infty}\begin{aligned}
	\|(-L)^{-\frac 1 2}(p_\epsilon(\cdot,y)-1)\|_\infty
	&\leq \frac 1 {\sqrt{\pi}}\int_0^\infty s^{-\frac 1 2}\sup_{x,y\in M}|p_{s+\epsilon}(x,y)-1|\,\dd s\\
	&\lesssim \int_0^\infty s^{-\frac 1 2}(s+2\epsilon)^{-\frac d 2}\eup^{-\lambda_1 s}\,\dd s\\
	&\lesssim
	\begin{cases}
	\log \epsilon^{-1}, &d=1,\\
	\epsilon^{-\frac{d-1}2}, &d\geq 2.
	\end{cases}
\end{aligned}\end{gather}
From $\int_{2\epsilon}^\infty e^{-\lambda_i t}\,\dd t = \lambda_i^{-1} \eup^{-2\lambda_i\epsilon}$ and Fubini's theorem we get
\begin{gather}\label{L2}\begin{aligned}
	\|(-L)^{-\frac 1 2}(p_\epsilon(\cdot,y)-1)\|_{L^2(\mu)}^2
	&=\sum_{i=1}^\infty\frac{\eup^{-2\lambda_i \epsilon}}{\lambda_i}\phi_i^2(y)\\
	&=\int_{2\epsilon}^\infty (p_t(y,y)-1) \,\dd t\\
	&\le\int_{2\epsilon}^\infty |p_t(y,y)-1| \,\dd t
	\lesssim
	\begin{cases}
	1, &d=1,\\
	\log(\epsilon^{-1}), &d=2,\\
	\epsilon^{-\frac {d-2} 2}, &d\ge3.
	\end{cases}
\end{aligned}\end{gather}

From the definition \eqref{sm} of $m(g)$ we see that $m(g)$ changes its behaviour if $d=2\alpha$. In order to bound $m(g)$ we will distinguish cases according to the dimension $d$ in relation to $\alpha$.

\medskip\noindent
\textbf{Case 1:} Assume that $d\leq 4\alpha$. In this case we can control $m(g)$ by $\|g\|_{L^2(\mu)}$,  and this expression can be relatively easily estimated. \normal

\smallskip
\emph{Case 1.1}: Assume that $\alpha=\frac 1 2$ and $d=1$. Using \eqref{S}, we have for any $ p > 2$
\begin{align*}
	m((-L)^{-\frac 1 2}\left(p_\epsilon(\cdot,y)-1)\right)
	&\lesssim \|(-L)^{-\frac 1 2}(p_\epsilon(\cdot,y)-1)\|_{L^p(\mu)}\\
	&\lesssim \|(-L)^{-\frac{p+2}{4p}}(p_\epsilon(\cdot,y)-1)\|_{L^2(\mu)}.
\end{align*}
From the fact that $\lambda_i^{-u}=\Gamma(u)^{-1}\int_0^\infty s^{u-1}\eup^{-\lambda_i s}\,\dd s$ with $u=\frac{2+p}{2p}$, and \eqref{Heat} we get
\begin{align*}
	m((-L)^{-\frac 1 2}\left(p_\epsilon(\cdot,y)-1)\right)^2
	&\lesssim\sum_{i=1}^\infty \eup^{-2\lambda_i \epsilon}\lambda_i^{-\frac{p+2}{2p}}\phi_i^2(y)\\
	&\lesssim \int_0^\infty \sum_{i=1}^\infty \eup^{-\lambda_i(2\epsilon+s)}\phi_i^2(y)s^{\frac 1 p-\frac 1 2}\,\dd s\\
	&=\int_0^\infty (p_{2\epsilon+s}(y,y)-1)s^{\frac 1 p-\frac 1 2}\,\dd s\\
	&\lesssim \int_0^\infty (2\epsilon+s)^{-\frac 1 2} s^{\frac 1 p-\frac 1 2} \eup^{-\lambda_1 s}\,\dd s
	<\infty.
\end{align*}
Thus, $B_\epsilon\lesssim 1$ for any $\epsilon>0$. Setting $\epsilon=T^{-1}$ in \eqref{W-Tri} shows
\begin{gather}\label{R0}
	\Ee^\mu\left[\Ww_p^p(\mu_{T}^\alpha,\mu)\right]
	\lesssim \epsilon^{\frac p 2} + T^{-\frac p 2}A_\epsilon+T^{-p}B_\epsilon
	\lesssim T^{-\frac p 2} + T^{-p}\lesssim T^{-\frac p 2},\quad T\gg 1.
\end{gather}

\smallskip
\emph{Case 1.2}: Assume that $d\leq 4\alpha$ and $(\alpha,d)\neq(\frac 1 2,1)$.  As mentioned earlier, we have  $m(g)\leq\|g\|_{L^2(\mu)}$ for any $g\in L^2(\mu)$ with $\mu(g)=0$. Combining this with \eqref{L2}, we see that
\begin{gather}\begin{aligned}\label{B<}
	B_\epsilon
	&=\int_M m\left((-L)^{-\frac 1 2}(p_\epsilon(\cdot,y)-1)\right)^p\,\mu(\dd y)\\
	&\lesssim\|(-L)^{-\frac 1 2}(p_\epsilon(\cdot,y)-1)\|_{L^2(\mu)}^p
	\lesssim
	\begin{cases}
	1, &d=1,\\
	\log^{\frac p 2}(\epsilon^{-1}), &d=2,\\
	\epsilon^{-\frac p 4}, &d=3,\\
	\epsilon^{-\frac p 2}, &d=4.
	\end{cases}
\end{aligned}\end{gather}
If we choose $\epsilon=T^{-1}$, then we easily see that $T^{-p}B_\epsilon\lesssim T^{-\frac p 2}$ for any $d$ satisfying $d\leq 4\alpha \leq 4$. As $d\in\nat$ and $d<4\alpha$ with $\alpha\in(0,1]$, the $d=4\alpha$ is the same as $(\alpha,d)=(1,4)$. In view of \eqref{A}, for large enough $T >1$, we also have
\begin{gather*}
	T^{-\frac p 2}A_\epsilon\lesssim
	\begin{cases}
	T^{-\frac p 2}, &d<4\alpha,\\
	T^{-\frac p 2}\log^{\frac p 2} T,&d=4.
	\end{cases}
\end{gather*}
Inserting this into \eqref{W-Tri}, we obtain
\begin{gather}\label{R1}
	\Ee^\mu\left[\Ww_p^p(\mu_{T}^\alpha,\mu)\right]
	\lesssim \epsilon^{\frac p 2} + T^{-\frac p 2}A_\epsilon + T^{-p}B_\epsilon
	\lesssim
	\begin{dcases}
		T^{-\frac p 2}, & d<4\alpha\\
		T^{-\frac p 2}\log^{\frac p 2} T, & d=4
	\end{dcases},
	\quad  T\gg 1.
\end{gather}

\medskip\noindent
\textbf{Case 2:} Assume that $d> 4\alpha$ and $T\gg 1$. Interpolating between \eqref{infty} and \eqref{L2} shows
\begin{align*}
	m\left((-L)^{-\frac 1 2}(p_\epsilon(\cdot,y)-1)\right)^p
	&= \|(-L)^{-\frac 1 2}(p_\epsilon(\cdot,y)-1)\|_{L^{\frac d {2\alpha}}(\mu)}^p\\
	&\leq \left[\|(-L)^{-\frac 1 2}(p_\epsilon(\cdot,y)-1)\|_\infty^{\frac d {2\alpha}-2}
	\|(-L)^{-\frac 1 2}(p_\epsilon(\cdot,y)-1)\|_{L^2(\mu)}^2\right]^{\frac{2\alpha p}{d}}\\
	&
	\lesssim
	\begin{cases}
	\log^{p(1-4\alpha)}(\epsilon^{-1}), &d=1,\\
	\epsilon^{\alpha p-\frac p 2}\log^{\alpha p}(\epsilon^{-1}), &d=2,\\
	\epsilon^{-\frac 12 p[d-(1+2\alpha)]}, &d\geq 3.
	\end{cases}
\end{align*}
Now we choose $\epsilon$ in the following way:
\begin{gather*}
	\epsilon
	=
	\begin{cases}
	T^{-1}, &d\le2(1+\alpha),\\
	T^{-\frac 2 {d-2\alpha}}, &d>2(1+\alpha).
	\end{cases}
\end{gather*}
There are now three cases
\begin{itemize}
\item If $d\leq 2$, then
\begin{gather}
	T^{-p}B_\epsilon
	\lesssim
	\begin{cases}
		T^{-p}\log^{(1-4\alpha)p}T, &d=1,\\
		T^{-\frac p 2-\alpha p}\log^{\alpha p}T, &d=2,
	\end{cases}
	\quad\text{hence,}\quad
	T^{-p}B_\epsilon\lesssim T^{-\frac p 2}.
\end{gather}

\item If $3\leq d<2(1+\alpha)$, then
\begin{gather*}
	T^{-p}B_\epsilon
	\lesssim T^{\frac 12 p[d-(1+2\alpha)]-p}
	\lesssim T^{\frac 12 p[2(1+\alpha)-(1+2\alpha)]-p}
	\lesssim T^{-\frac p 2}.
\end{gather*}

\item If $3\leq d$  and $d=2(1+\alpha)$,  then we have either $(\alpha,d)=(\frac 1 2,3)$ or $(\alpha,d)=(1,4)$. We have already considered the case $(\alpha,d)=(1,4)$ further up.  A direct calculation reveals that $T^{-p}B_\epsilon\lesssim T^{-\frac p 2}$ if  $(\alpha,d)=(\frac 1 2,3)$.

\item If $3\leq d$ and $d>2(1+\alpha)$, then
\begin{gather*}
	T^{-p}B_\epsilon
	\lesssim T^{\frac{(d-(1+2\alpha))p}{d-2\alpha}-p}
	= T^{-\frac p {d-2\alpha}}.
\end{gather*}
\end{itemize}
We can combine these cases to see that under the assumption $d>4\alpha$,
\begin{gather*}
	T^{-p}B_\epsilon
	\lesssim
	\begin{cases}
	T^{-\frac p 2}, &d<2(1+\alpha),\\
	T^{-\frac p 2}\log^{\frac p 2} T, &d=2(1+\alpha),\\
	T^{-\frac p {d-2\alpha}}, &d>2(1+\alpha).
	\end{cases}
\end{gather*}
Because of our choice of $\epsilon$, one can also easily check that $T^{-\frac p 2}A_\epsilon$ enjoys the same upper bounds as $T^{-p}B_\epsilon$. Thus, by \eqref{W-Tri}, we have
\begin{gather}\begin{aligned}\label{R3}
	\Ee^\mu\left[\Ww_p^p(\mu_{T}^\alpha,\mu)\right]
	&\lesssim \epsilon^{\frac p 2}+T^{-\frac p 2}A_\epsilon+T^{-p}B_\epsilon\\
	&\lesssim
	\begin{cases}
	T^{-\frac p 2}, &d<2(1+\alpha),\\
	T^{-\frac p 2}\log^{\frac p 2}T, &d=2(1+\alpha),\\
	T^{-\frac p {d-2\alpha}}, &d>2(1+\alpha).
	\end{cases}
\end{aligned}\end{gather}
This completes Case 2.

Combining  \eqref{R1} and \eqref{R3}, finally finishes the proof.
\end{proof}

\begin{proof}[Proof of Theorem \ref{Wpq}]
Let $q>0$ and $p\geq 1$. Since the Wasserstein distance is monotone, the H\"{o}lder inequality shows
\begin{align*}
	\Ee^{\nu}\left[\Ww_p^q(\mu_T^\alpha,\mu)\right]
	&=\int_{M} h(x) \Ee^{x}\left[\Ww_p^q(\mu_T^\alpha,\mu)\right]\mu(\dd x)\\
	&\leq \int_{M} h(x) \Ee^{x}\left[\Ww_{p\vee q}^q(\mu_T^\alpha,\mu)\right]\mu(\dd x)\\
	&\leq \int_M h(x)\left(\Ee^x\left[\Ww_{p\vee q \vee 1}^{p\vee q}(\mu_T^\alpha,\mu)\right]\right)^{\frac q {p\vee q}}\,\mu(\dd x)\\
	&\leq \|h\|_{ L^{q^*}(\mu)} \Ee^\mu\left[\Ww_{p\vee q}^{p\vee q}(\mu_T^\alpha,\mu)\right]^\frac{q}{{p\vee q}}\\
	&\lesssim \|h\|_{ L^{q^*}(\mu)} \gamma_{\alpha,d}(T)^q,
\end{align*}
where $q^*$ is the conjugate of the exponent $(p\vee q)/q$; the last inequality is from Lemma \ref{pp}.
\end{proof}

\section{The lower bound -- proof of Theorem~\ref{Wpq-L}}

We split the proof of Theorem~\ref{Wpq-L} in a series of lemmas.

\begin{lemma}\label{Gradient}
\begin{enumerate}
\item\label{Gradient-a}
Let $Z\neq 0$ and $\alpha\in(\frac 12,1]$ or $Z=0$ and $\alpha\in(0,1]$. Then \normal
\begin{gather}\label{Poin}
	\|P_t^\alpha-\mu\|_{L^2(\mu)}
	\leq \eup^{-\lambda_1^\alpha t},\quad t>0.
\end{gather}
\item\label{Gradient-b}
If $Z\neq 0$, then there exists for any $\alpha\in(1/2,1]$ a constant $c>0$ such that
\begin{gather}\label{G-A}
	\|\nabla P_t^\alpha f\|_{2}
	\leq c(1\wedge t)^{-\frac 1 {2\alpha}}\eup^{-\lambda_1^\alpha t}\|f\|_{L^2(\mu)},
	\quad t>0,\;f\in C_{b,L}(M).
\end{gather}
\end{enumerate}
\end{lemma}

\begin{proof}
\ref{Gradient-a}
If $\alpha=1$, the claim follows from \eqref{ULT}. Since $\lambda_1>0$ is the spectral gap of $L$, \normal it is clear that $\lambda_1^\alpha>0$ is the spectral gap of $(-L)^\alpha$. By \cite[Theorem 1.1.1]{W05}, this is equivalent to the following Poincar\'e inequality  for $f\in\Dscr(\Escr^\alpha)$ satisfying $\mu(f)=0$:
\begin{gather*}
	\mu(f^2)
	\leq \frac 1 {\lambda_1^\alpha} \left\langle (-L)^\alpha f,f\right\rangle_{L^2(\mu)}
	= \frac 1 {\lambda_1^\alpha}\left\langle (-\Lcal)^\alpha f,f\right\rangle_{L^2(\mu)}
	= \frac 1 {\lambda_1^\alpha}\Escr^\alpha(f,f).
\end{gather*}
Using once again \cite[Theorem 1.1.1]{W05}, we have
\begin{gather}\label{Pin}
	\|P_t^\alpha-\mu\|_{L^2(\mu)}\leq \eup^{-\lambda_1^\alpha t}.
\end{gather}

\medskip\noindent
\ref{Gradient-b} Assume that $Z\neq 0$. An elementary calculation shows that for any $\alpha\in(1/2,1]$ we have $x\eup^{-x^\alpha}\leq c_1$ for some constant $c_1=c_1(\alpha)$ and all $x\in\real$. \normal Using the spectral representation of $f\in L^2(\mu)$,  we have for the semigroup $(\hat P_t^\alpha)_{t\geq 0}$ with generator $-(-L)^\alpha$ \normal
\begin{gather*}
	\|\nabla \widehat{P}_t^\alpha f\|_{L^2(\mu)}^2
	= \sum_{i=1}^\infty \lambda_i\eup^{-2\lambda_i^\alpha t}\mu(f\phi_i)^2
	\leq c_1 t^{-\frac 1 {\alpha}}\eup^{-\lambda_1^\alpha t}\|f\|_{L^2(\mu)}^2,
	\quad t>0.
\end{gather*}
With Duhamel's formula we see
\begin{align*}
	\|\nabla P_t^\alpha f\|_{L^2(\mu)}
	&\leq \|\nabla \widehat{P}_t^\alpha f\|_{L^2(\mu)}+\int_0^t\|\nabla \widehat{P}_s^\alpha\{Z P_{t-s}^\alpha f\}\|_{L^2(\mu)} \,\dd s\\
	&\leq c_1 t^{-\frac 1 {2\alpha}}\|f\|_{L^2(\mu)}+c_1\|Z\|_\infty \int_0^t s^{-\frac 1 {2\alpha}} \|\nabla P_{t-s}^\alpha f\|_{L^2(\mu)} \,\dd s,
	\quad t>0.
\end{align*}
We can now use the generalized Gronwall inequality from \cite{YGD}, to get
\begin{gather}\label{t<1}
	\|\nabla P_t^\alpha f\|_{L^2(\mu)}
	\leq c_2 t^{-\frac 1 {2\alpha}}\|f\|_{L^2(\mu)},\quad t\leq 1,
\end{gather}
for some constant $c_2>0$. In order to deal with $t\geq 1$, we use the semigroup property, \eqref{t<1} and \eqref{Poin}: there is some constant  $c_3>0$ such that we have for $t\geq 1$
\begin{gather}\label{t>1}\begin{aligned}
	\|\nabla P_t^\alpha f\|_{L^2(\mu)}
	=\|\nabla P_t^\alpha (f-\mu(f))\|_{L^2(\mu)}
	&=\|\nabla P_1^\alpha P_{t-1}^\alpha (f-\mu(f))\|_{L^2(\mu)}\\
	&\leq c_2\|P_{t-1}^\alpha (f-\mu(f))\|_{L^2(\mu)}\\
	&\leq c_3\eup^{-\lambda_1^\alpha t}\|f\|_{L^2(\mu)}.
\end{aligned}\end{gather}
The claim follows from \eqref{t<1} and \eqref{t>1}.
\end{proof}

Recall from \eqref{Sigma} that $\Sigma(f) = \int_0^\infty \mu(fP_t^\alpha f)\,\dd t$.
\begin{lemma}\label{V}
	Let $\alpha\in (\frac 12,1]$ if $Z\neq 0$ and $\alpha\in (0,1]$ if $Z= 0$. For any $i\geq 1$, we have
\begin{gather*}
	\Sigma(\phi_i)=\lambda_i^{-\alpha}-\lambda_i^{-2\alpha}\Sigma(Z\phi_i).
\end{gather*}
\end{lemma}
\begin{proof}
	Since $\Lcal^\alpha$ is the generator of the semigroup $(P_t^\alpha)_{t\geq 0}$ and $\phi_i\in\Dscr(\Lcal^\alpha)\cap\Dscr((-L)^\alpha)$, we have
	\begin{gather*}
		P_t^\alpha\phi_i - \phi_i
		= \int_0^t P_s^\alpha \Lcal^\alpha\phi_i\,\dd s
		= -\int_0^t P_s^\alpha (-L)^\alpha\phi_i\,\dd s + \int_0^t P_s^\alpha Z\phi_i\,\dd s.
	\end{gather*}
	If we multiply this equality with $g\in L^2(\mu)$ and integrate w.r.t.\ $\mu$, we get
	\begin{gather}\label{eq-new}
		\mu(g P_t^\alpha\phi_i) - \mu(g\phi_i)
		= -\lambda_i^\alpha \int_0^t \mu(gP_s^\alpha\phi_i)\,\dd s + \int_0^t \mu(gP_s^\alpha Z\phi_i)\,\dd s.
	\end{gather}
	Observe that because of \eqref{Poin}
	\begin{gather*}
		\mu(gP_t^\alpha\phi_i)
		\leq \|g\|_{L^2(\mu)}\|P_t^\alpha\phi_i\|_{L^2(\mu)}
		\leq \|g\|_{L^2(\mu)}\eup^{-\lambda_1^\alpha t}
		\xrightarrow{t\to\infty} 0.
	\end{gather*}
	Letting $t\to\infty$ turns \eqref{eq-new} into
	\begin{gather}\label{eq:new-2}
		\lambda_i^\alpha \int_0^\infty \mu(gP_s^\alpha\phi_i)\,\dd s
		=\mu(g\phi_i) + \int_0^\infty \mu(gP_s^\alpha Z\phi_i)\,\dd s.
	\end{gather}
	Now we take $g=Z\phi_i$ and observe that $Z$ is divergence-free:
	\begin{gather*}
	\lambda_i^\alpha \int_0^\infty \mu(Z\phi_i P_s^\alpha\phi_i)\,\dd s
		= \int_0^\infty \mu(Z\phi_i P_s^\alpha Z\phi_i)\,\dd s.
	\end{gather*}
	The equality \eqref{eq:new-2} remains true for the adjoint semigroup $P_t^{\alpha,*}$ (which is generated by $-(-L)^\alpha-Z$), if we replace $Z$ by $-Z$. Taking $g=\phi_i$ in the adjoint version of \eqref{eq:new-2} and observing that $\mu(\phi_i^2)=1$, we get
	\begin{align*}
		\lambda_i^\alpha \int_0^\infty \mu(\phi_iP_s^\alpha \phi_i)\,\dd s
		= \lambda_i^\alpha \int_0^\infty \mu(\phi_iP_s^{\alpha,*}\phi_i)\,\dd s
		&= 1 - \int_0^\infty \mu(\phi_i P_s^{\alpha,*} Z\phi_i)\,\dd s\\
		&= 1 - \int_0^\infty \mu(Z\phi_i P_s^{\alpha} \phi_i)\,\dd s.
	\end{align*}
	If we combine the last two formulae, the claim follows.
\end{proof}

For the next lemma we need the following functional of the Markov process $(\Xcal_t^\alpha)_{t\geq 0}$:
\begin{gather}\label{eq:psialpha}
	\psi_i^\alpha(T) := \frac 1 {\sqrt{T}}\int_0^T \phi_i(\Xcal_t^\alpha)\,\dd t, \quad i\in\nat.
\end{gather}
\begin{lemma}\label{PHI}
\begin{enumerate}
\item\label{PHI-a}
	Assume that $Z=0$ and $\alpha\in(0,1]$. There exists some constant $c>0$ such that
	\begin{gather}\label{phiL-0}
		\left|\Ee^\mu\left[|\psi_i^\alpha(T)|^2\right] -\frac 2 {\lambda_i^\alpha}\right|
		\leq \frac{c}{\lambda_i^{2\alpha}T},\quad i\geq 1,\;  T \gg 1.
	\end{gather}

\item\label{PHI-b}
	Assume that $Z\neq 0$ and $\alpha\in(1/2,1]$. There exist constants $c,c'>0$ such that
	\begin{gather}\label{phiL}
		\left|\Ee^\mu\left[|\psi_i^\alpha(T)|^2\right]-\frac 2 {\lambda_i^\alpha}+\frac 2{\lambda_i^{2\alpha}}\Sigma(Z\phi_i)\right|
		\leq \frac c{\lambda_i^{2\alpha} T},\quad i\geq 1,\;T\gg 1,
	\end{gather}
	and
	\begin{gather*}
		\Sigma(Z\phi_i)\leq c' \lambda_i^{\frac 1 2}.
	\end{gather*}
\end{enumerate}
\end{lemma}
\begin{proof}
\ref{PHI-a}
	Let $Z=0$ and $\alpha\in (0,1]$. From the Markov property and the fact  that $P_t^\alpha$ is invariant with respect to $\mu$ we get
\begin{align*}
	\Ee^\mu\left[|\psi_i(T)|^2\right]
	&= \Ee^\mu\left[\frac 2T\int_0^T \dd t\int_{t}^T  \phi_i(\Xcal_s^\alpha)\,\dd s \: \phi_i(\Xcal_t^\alpha)\right]\\
	&= \Ee^\mu\left[\frac 2T\int_0^T\dd t\int_{t}^T \Ee^{\Xcal_t^\alpha} \phi_i(\Xcal_{s-t}^\alpha)\,\dd s \: \phi_i(\Xcal_t^\alpha)\right]\\
	&= \frac 2 T\int_0^T\dd t\int_{t}^T\mu(\phi_i P_{s-t}^\alpha\phi_i)\,\dd s.
\end{align*}
Since $\mu(\phi_i^2)=1$ and $P_t^\alpha \phi_i=\eup^{-\lambda_i^\alpha t}\phi_i$ this becomes
 \begin{align*}
 	\Ee^\mu\left[|\psi_i(T)|^2\right]
	=\frac 2 T\int_0^T \,\dd t\int_{t}^T \eup^{-\lambda_i^{\alpha}(s-t)}\,\dd s
	&=\frac 2 T \int_0^T\frac 1 {\lambda_i^\alpha}\left(1-\eup^{-\lambda_i^\alpha(T-t)}\right)\dd t\\
	&=\frac 2 {\lambda_i^\alpha}-\frac 2 {T\lambda_i^{2\alpha}}\left(1-\eup^{-\lambda_i^\alpha T}\right).
\end{align*}
Therefore, there exists some constant $c_1>0$ such that for any  $T \ge 1$,
\begin{gather*}
	\left|\Ee^\mu[|\psi_i(T)|^2]-\frac 2 {\lambda_i^\alpha}\right|
	\leq \frac{c_1}{\lambda_i^{2\alpha}T}.
\end{gather*}

\medskip\noindent
\ref{PHI-b} Let $Z\neq 0$ and $\alpha\in(\frac 12,1]$. Similar to the first calculation in Part~\ref{PHI-a} we get
\begin{align*}
	\Ee^\mu\left[|\psi_i^\alpha(T)|^2\right]
	&= \frac 2 T \int_0^T\,\dd t\int_{t}^T\mu(\phi_i P_{s-t}^\alpha\phi_i)\,\dd s
	= \frac 2 T\int_0^T\,\dd t\int_0^{T-t}\mu(\phi_i P_s^\alpha\phi_i)\,\dd s\\
	&=\frac 2 T\int_0^T\,\dd t\int_0^\infty\mu(\phi_i P_s^\alpha\phi_i)\,\dd s-\frac 2 T \int_0^T\,\dd t\int_{T-t}^\infty \mu(\phi_i P_s^\alpha \phi_i) \,\dd s.
\end{align*}
Using Lemma \ref{V} shows
\begin{gather}\begin{aligned}\label{phiT}
	\Ee^\mu\left[|\psi_i^\alpha(T)|^2\right]
	&=2\Sigma(\phi_i)-\frac 2 T \int_0^T\,\dd t\int_{T-t}^\infty \mu(\phi_i P_s^\alpha \phi_i) \,\dd s\\
	&=2\left(\frac 1 {\lambda_i^\alpha}-\frac 1 {\lambda_i^{2\alpha}}\Sigma(Z\phi_i)\right)
	- \frac 2 T \int_0^T\,\dd t\int_{T-t}^\infty \mu(\phi_i P_s^\alpha \phi_i) \,\dd s.
\end{aligned}\end{gather}
We will now estimate the remaining integral. The Cauchy-Schwarz inequality and the fact that $\|P_s^\alpha\phi_i\|_{L^2(\mu)}\leq \eup^{-\lambda_i^\alpha s}$, cf.\ \eqref{Poin},  gives
\begin{align*}
	\left|\frac 2 T \int_0^T\,\dd t\int_{T-t}^\infty \mu(\phi_i P_s^\alpha \phi_i) \,\dd s\right|
	&\leq \frac 2 T \int_0^T\,\dd t\int_{T-t}^\infty e^{-\lambda_i^\alpha s} \,\dd s
\le \frac 2 {\lambda_i^{2\alpha}T},\quad T \ge 1.
\end{align*}
This proves \eqref{phiL}.

From \eqref{G-A} and $\|\nabla\phi_i\|_{L^2(\mu)}^2 = \left|\langle -L \phi_i,\phi_i\rangle_{L^2(\mu)}\right| = \lambda_i$, we have
\begin{align*}
	\left|\mu\left((Z\phi_i)P_t^\alpha (Z\phi_i)\right)\right|
	=\left|\mu\left(\phi_i(Z P_t^\alpha(Z\phi_i))\right)\right|
	&\leq \|Z\|_\infty\|\nabla P_t^\alpha(Z\phi_i)\|_{L^2(\mu)}\\
	&\leq c_5\|Z\|_\infty(1\wedge t)^{-\frac 1 {2\alpha}}\eup^{-\lambda_1^\alpha t}\|Z\phi_i\|_{L^2(\mu)}\\
	&\leq  c_5 \|Z\|_\infty^2\sqrt{\lambda_i}(1\wedge t)^{-\frac 1 {2\alpha}}\eup^{-\lambda_1^\alpha t},
\end{align*}
with a suitable constant $c_5>0$. Therefore, we get from the definition \eqref{Sigma} of $\Sigma(\cdot)$ that
\begin{gather*}
	\Sigma(Z\phi_i)\leq c_6 \lambda_i^{\frac 1 2},
\end{gather*}
for some constant $c_6>0$.
\end{proof}

The following lemma is originally from \cite[Proposition 5.3]{WZ}. For the convenience of our readers, we present the suitable adapted argument here. \normal Recall that $f_{T,\epsilon}^\alpha(\cdot)=\frac 1 T \int_0^T p_\epsilon(\Xcal_t^\alpha,\cdot)\,\dd t$.
\begin{lemma}
Let
\begin{gather}
	\label{g-te}g_{T,\epsilon}^\alpha
	:= (-L)^{-1}(f_{T,\epsilon}^\alpha-1),
\end{gather}
then for any $\beta>0$,
\begin{gather}\label{W1-L}\begin{aligned}
	\Ee^\mu\left[\Ww_1(\mu_{T,\epsilon}^\alpha,\mu)\right]
	&\geq
	\beta^{-1}\Ee^\mu\left[\|\nabla g_{T,\epsilon}^\alpha\|_{L^2(\mu)}^2\right] \\
	&\mbox{}\quad- \beta^{-\frac 3 2}\left(\Ee^\mu\left[\|\nabla g_{T,\epsilon}^\alpha\|_{L^4(\mu)}^4\right]\right)^{\frac 1 4}
	\left(\Ee^\mu\left[\|\nabla g_{T,\epsilon}^\alpha\|_{L^2(\mu)}^2\right]\right)^{\frac 3 4}.
\end{aligned}\end{gather}
\end{lemma}
\begin{proof}
Since $M$ is compact, the Ricci curvature is bounded and we can use \cite[Corollary~1]{RS} to see that there exists some  constant $K\in\real$ such that
\begin{gather*}
	\Ww_1(\mu,\nu P_\epsilon)
	\leq \eup^{K \epsilon}\Ww_1(\mu,\nu),
	\quad\text{for all\ \ } \epsilon>0 \text{\ \ and\ \ } \nu\in\Pscr.
\end{gather*}
We will later on pick a suitable $\epsilon$, cf.\ \eqref{eps-}, but we may always assume that $\epsilon\in (0,1)$, thus $\eup^{K\epsilon}$ is bounded and bounded away from zero.  Therefore, it  is enough to find a lower bound for $\Ee^\mu\left[\Ww_1(\mu_{T,\epsilon}^\alpha,\mu)\right]$.
According to Lusin's approximation theorem (see \cite[Lemma 5.2(2)]{WZ}), there exists some constant $c_1>0$ such that for any $\beta>0$, there exists a function $g_{\beta}^\alpha$ with Lipschitz constant $\beta$, such that
\begin{equation}\label{mE}
	\mu(E_\beta)\leq \frac{c_1}{\beta^2}\int_M|\nabla g_{T,\epsilon}^\alpha|^2\,\dd\mu,
\end{equation}
where
\begin{gather*}
	E_\beta:=\left\{x\in M\::\: g_{T,\beta}^\alpha(x)\neq g_{\beta}^\alpha(x)\right\}.
\end{gather*}
By Rademacher's theorem, $\nabla g_\beta^\alpha$ exists Lebesgue a.e., hence  $\|\nabla g_{\beta}^\alpha\|_{ L^\infty(\mu)}\leq \beta$. Using Kantorovich's dual formula for $\Ww_1$, we get
\begin{equation}\begin{split}\label{W1L}
	\Ww_1\left(\mu_{T,\epsilon}^\alpha,\mu\right)
	&\geq \beta^{-1}\left|\mu_{T,\epsilon}^\alpha(g_\beta^\alpha)-\mu(g_{\beta}^\alpha)\right|\\
	&=\beta^{-1}\left|\int_M g_\beta^\alpha(f_{T,\epsilon}^\alpha-1)\,\dd\mu\right|\\
    &=\beta^{-1}\left|\int_M g_\beta^\alpha(-L)(g_{T,\epsilon}^\alpha)\,\dd\mu\right|\\
    &\ge \beta^{-1}\left|\int_M g_{T,\epsilon}^\alpha (-L)(g_{T,\epsilon}^\alpha)\,\dd\mu\right|-
    \beta^{-1}\left|\int_M (g_\beta^\alpha-g_{T,\epsilon}^\alpha) (-L)(g_{T,\epsilon}^\alpha)\,\dd\mu\right|\\
    &=\beta^{-1} \|\nabla g_{T,\epsilon}^\alpha\|_{L^2(\mu)}^2-
    \beta^{-1}\left|\int_M (g_\beta^\alpha-g_{T,\epsilon}^\alpha) (-L)(g_{T,\epsilon}^\alpha)\,\dd\mu\right|.
\end{split}\end{equation}
By construction, $|\nabla(g_\beta^\alpha-g_{T,\epsilon}^\alpha)|=0$ on $E_\beta^c$. Since $g_\beta^\alpha$ is a Lipschitz function with Lipschitz constant $\beta$ such that \eqref{mE} holds, we see
\begin{align*}
	\left|\int_M (g_\beta^\alpha-g_{T,\epsilon}^\alpha) (-L)(g_{T,\epsilon}^\alpha)\,\dd\mu\right|
	&=\left|\int_M \nabla(g_\beta^\alpha-g_{T,\epsilon}^\alpha) \nabla g_{T,\epsilon}^\alpha\,\dd\mu\right|\\
	&\le\mu\left(\I_{E_\beta}|\nabla g_{T,\epsilon}^\alpha|^2\right) + \beta\mu\left(\I_{E_\beta}|\nabla g_{T,\epsilon}^\alpha|\right)\\
	&\leq \mu(E_\beta)^{\frac 1 4}\|\nabla g_{T,\epsilon}^\alpha\|_{L^4(\mu)}\|\nabla g_{T,\epsilon}^\alpha\|_{L^2(\mu)}
	+ \beta\mu(E_{\beta})^{\frac 3 4}\|\nabla g_{T,\epsilon}^\alpha\|_{L^4(\mu)}\\
	&\leq c\beta^{-\frac 1 2}\|\nabla g_{T,\epsilon}^\alpha\|_{L^4(\mu)}\|\nabla g_{T,\epsilon}^\alpha\|_{L^2(\mu)}^{\frac 3 2}.
\end{align*}
From this and \eqref{W1L}, we obtain
\begin{gather*}
	\Ww_1\left(\mu_{T,\epsilon}^\alpha,\mu\right)
	\geq \beta^{-1}\|\nabla g_{T,\epsilon}^\alpha\|_{L^2(\mu)}^2
	- c\beta^{-\frac 3 2}\|\nabla g_{T,\epsilon}^\alpha\|_{L^4(\mu)}\|\nabla g_{T,\epsilon}^\alpha\|_{L^2(\mu)}^{\frac 3 2}.
\end{gather*}
An applicatoin of H\"older's inequality finishes the proof.
\end{proof}

\begin{lemma}\label{mu-L}
	Assume that $\alpha\in(\frac 12,1]$. For any  $q>0$, one has
	\begin{gather}\label{eq-mu-L}
		\Ee^\mu\left[\Ww_1^q(\mu_{T,\epsilon}^\alpha,\mu)\right]
		\gtrsim \gamma_{\alpha,d}(T)^q,\quad T\gg 1.
	\end{gather}
\end{lemma}
\begin{proof}
Let us first show that it is enough to prove \eqref{eq-mu-L} for $q=1$. Indeed:
\begin{itemize}
\item
if $q\geq 1$, then we can use Jensen's inequality to get
\begin{gather*}
	\Ee^\mu \left[\Ww_1^q(\mu_T^\alpha, \mu)\right]
	\geq \Ee^\mu \left[\Ww_1(\mu_T^\alpha, \mu)\right]^q.
\end{gather*}
\item
if $q \in (0, 1)$, then we can use the log-convexity of the norm $p\mapsto \left(\Ee \Ww_1^p(\mu_T^\alpha, \mu)\right)^{1/p}$, $p\in (0,\infty)$, for the convex-combination $q\frac 1{2-q}  + 2\frac{1-q}{2-q}=1$ to get
\begin{gather*}
	\Ee^\mu\left[\Ww_1^q(\mu_T^\alpha, \mu)\right]
	\geq \Ee^\mu \left[\Ww_1^2(\mu_T^\alpha, \mu)\right]^{-(1 - q)} \cdot \Ee^\mu\left[\Ww_1(\mu_T^\alpha, \mu)\right]^{2 - q}.
\intertext{Combining this with the upper bound from Theorem~\ref{Wpq} yields}
	\Ee^\mu\left[\Ww_1^q(\mu_T^\alpha, \mu)\right]
	\gtrsim \gamma_{\alpha,d}(T)^{-2 (1 - q)}\Ee^\mu\left[\Ww_1(\mu_T^\alpha, \mu)\right]^{2 - q}.
\end{gather*}
\end{itemize}
So, in either case it is enough to prove $\Ee^\mu \left[\Ww_1(\mu_T^\alpha, \mu)\right]\gtrsim \gamma_{\alpha,d}(T)$. Because of \eqref{W1-L}, this can be done by suitable moment estimates of $\|\nabla g_{T,\epsilon}^\alpha\|$.

Recall the definition of $g_{T,\epsilon}^\alpha$, $f_{T,\epsilon}^\alpha$ and $\psi_i^\alpha(T)$, cf.\ \eqref{g-te}, \eqref{eq:ftalpha} and \eqref{eq:psialpha}, respectively. We use the formula \eqref{HS2} for the heat kernel $p_t(x,y)$ in terms of the orthonormal basis $\{\phi_i\}_{i\in\nat}$ of eigenfunctions of $-(-L)$. In view of the the integration by parts formula \eqref{eq:ibp} we get
\begin{equation}\begin{split}\label{E-g}
	\Ee^\mu\left[\|\nabla g_{T,\epsilon}^\alpha\|_{L^2(\mu)}^2\right]
	&=\Ee^\mu \left[\langle (-L) g_{T,\epsilon}^\alpha, g_{T,\epsilon}^\alpha\rangle_{L^2(\mu)}\right]\\
	&=\Ee^\mu \left[\langle (-L)^{-1} (f_{T,\epsilon}^\alpha-1), (f_{T,\epsilon}^\alpha-1) \rangle_{L^2(\mu)}\right]\\
	&=\Ee^\mu \left[\left\langle \sum_{i=1}^\infty \eup^{-\lambda_i \epsilon}\mu_T^\alpha(\phi_i)\phi_i, \sum_{i=1}^\infty \eup^{-\lambda_i \epsilon}\lambda_i^{-1}\mu_T^\alpha(\phi_i)\phi_i\right\rangle_{L^2(\mu)}\right]\\
	&=\sum_{i=1}^\infty \frac{\eup^{-2\lambda_i \epsilon}}{\lambda_i}
	\Ee^\mu \left[\left|\mu_T^\alpha(\phi_i)\right|^2\right]\\
	&=\frac 1T\sum_{i=1}^\infty \frac{\eup^{-2\lambda_i\epsilon}}{\lambda_i} \Ee^\mu\left[\left|\psi_i^\alpha(T)\right|^2\right].
\end{split}\end{equation}
From Lemma~\ref{PHI}~\ref{PHI-b} we know that
\begin{gather}\label{*****}
	\Ee^\mu\left[|\psi_i^\alpha(T)|^2\right]
	\geq \frac 2 {\lambda_i^\alpha} \left(1-c\lambda_i^{\frac 1 2-\alpha}\right) - \frac c {\lambda_i^{2\alpha} T}
	\geq\frac 1 {\lambda_i^\alpha}-\frac c {\lambda_i^{2\alpha}T}. 
\end{gather}
The last estimate uses that $\left(1-c\lambda_i^{\frac 1 2-\alpha}\right)\geq \frac 12$, which is true for sufficiently large $i\geq m$ since $\alpha> 1 /2$ and $\lambda_i$ is increasing with respect to $i$.

Inserting \eqref{*****} into \eqref{E-g} gives for any $T\ge 2c \lambda_1^{-\alpha}\ge 2c \lambda_i^{-\alpha}$ and $i\geq m$
\begin{align}\label{ng-l}
	\Ee^\mu\left[\|\nabla g_{T,\epsilon}^\alpha\|_{L^2(\mu)}^2\right]
	&\geq
	\frac 1 {2T} \sum_{i=m}^\infty\frac {\eup^{-2\lambda_i \epsilon}}{\lambda_i^{1+\alpha}}.\normal
\intertext{ Since $\lambda_i\sim i^{\frac 2d}$, we get}
\notag
	\frac 1 T \sum_{i=m}^\infty\frac {\eup^{-2\lambda_i \epsilon}}{\lambda_i^{1+\alpha}}
	&\gtrsim
	\frac 1 T
	\begin{cases}
		1,&d<2(1+\alpha),\\
		\log(1+\epsilon^{-1}),& d=2(1+\alpha),\\
		\epsilon^{1+\alpha-\frac d 2},& d>2(1+\alpha).\\
	\end{cases}
\intertext{If we take}
\label{eps-}
	\epsilon
	&=
	\begin{cases}
		T^{-1},& d\leq 2(1+\alpha),\\
		T^{-\frac 2 {d-2\alpha}}, &d>2(1+\alpha),
\end{cases}
\intertext{then we see}
\label{SI-1}
	\frac 1 T \sum_{i=m}^\infty\frac {\eup^{-2\lambda_i \epsilon}}{\lambda_i^{1+\alpha}}
	&\gtrsim
	\begin{cases}
		T^{-1}, & d<2(1+\alpha),\\
		T^{-1}\log  T, &d=2(1+\alpha),\\
		T^{-\frac 2 {d-2\alpha}}, &d>2(1+\alpha),
	\end{cases}
\end{align}
Combining \eqref{ng-l} and \eqref{SI-1} we arrive at
\begin{gather}\label{ng-2}
	\Ee^\mu\left[\|\nabla g_{T,\epsilon}^\alpha\|_{L^2(\mu)}^2\right]
	\gtrsim
	\begin{cases}
		T^{-1}, & d<2(1+\alpha),\\
		T^{-1}\log T, &d=2(1+\alpha),\\
		T^{-\frac 2 {d-2\alpha}}, &d>2(1+\alpha).
	\end{cases}
\end{gather}
Due to the boundedness of the Riesz transform we see that
\begin{gather*}
	\|\nabla(-L)^{-\frac 1 2} f\|_{L^p(\mu)}
	\leq k(p)\|f\|_{L^p(\mu)},
	\quad f\in L^p(\mu),\;\mu(f)=0,
\end{gather*}
where $k(p)$ is a constant depending on $p$. Thus,
\begin{gather*}
	\Ee^\mu\left[\|\nabla g_{T,\epsilon}^\alpha\|_{L^4(\mu)}^4\right]
	\lesssim \Ee\left[\|(-L)^{-\frac 1 2}(f_{T,\epsilon}^\alpha-1)\|_{L^4(\mu)}^4\right].
\end{gather*}
Choosing $\epsilon$ as in \eqref{eps-},  setting $p=4$ in \eqref{R1}, and using  H\"older's inequality, we have
\begin{gather}\label{g4}
	\Big(\Ee^\mu\left[\|\nabla g_{T,\epsilon}^\alpha\|_{L^2(\mu)}^2\right]\Big)^2\le\Ee^\mu\left[\|\nabla g_{T,\epsilon}^\alpha\|_{L^4(\mu)}^4\right]
	\lesssim
	\begin{cases}
		T^{-2}, &d<2(1+\alpha),\\
		T^{-2}\log^{ 2}T, &d=2(1+\alpha),\\
		T^{-\frac 4 {d-2\alpha}}, &d>2(1+\alpha).
	\end{cases}
\end{gather}
From  \eqref{g4}, for large enough $T>0$,
\begin{gather*}
	\left(\Ee^\mu\left[\|\nabla g_{T,\epsilon}^\alpha\|_{L^4(\mu)}^4\right]\right)^{\frac 1 4}
	\left(\Ee^\mu\left[\|\nabla g_{T,\epsilon}^\alpha\|_{L^2(\mu)}^2\right]\right)^{\frac 3 4}
	\lesssim
	\begin{cases}
		T^{-\frac 5 4}, & d<2(1+\alpha),\\
		T^{-\frac 5 4}\log^{\frac 5 4} T, & d=2(1+\alpha),\\
		T^{-\frac 5 {2(d-2\alpha)}}, & d>2(1+\alpha).
	\end{cases}
\end{gather*}
Plugging this into \eqref{W1-L} yields
\begin{gather*}
	\Ee^\mu\left[\Ww_1(\mu_{T,\epsilon}^\alpha,\mu)\right]
	\gtrsim \beta^{-1}\gamma_{\alpha,d}(T)^2-\beta^{-\frac 3 2}\gamma_{\alpha,d}(T)^{\frac 5 2},\quad T\gg 1.
\end{gather*}
Choosing $\beta = K\gamma_{\alpha,d}(T)$ for a sufficiently large constant $K>0$ finishes the proof.
\end{proof}

\begin{proof}[Proof of Theorem~\ref{Wpq-L}]
 Recall that $\alpha\in (\frac 12,1]$ if $Z\neq 0$ and $\alpha\in (0,1]$, if $Z=0$. Denote by $(S_t^\alpha)_{t\geq 0}$ the $\alpha$-stable subordinator. Using the scaling property of stable processes, we have $\textrm{law} \left(S_t^{\alpha}\right) = t^{1/\alpha} \textrm{law} \left(S_1^{\alpha}\right)$, hence $\Ee \left[\left(S_t^\alpha\right)^{-d/2}\right] = t^{-\frac d{2\alpha}}\Ee \left[\left(S_1^\alpha\right)^{-d/2}\right] = ct^{-\frac d{2\alpha}}$, see also Deng et al.~\cite[Section 2]{deng-et-al} for further details on (negative) moments of subordinators.  Thus,
\begin{gather}\label{ULT-1}
	\|\widehat{P}_t^\alpha-\mu\|_{1\to\infty}
	= \Ee\|P_{S_t^\alpha}-\mu\|_{1\to\infty}
	\lesssim \Ee\left[(S_t^\alpha)^{-\frac d 2}\right]
	\lesssim t^{-\frac d {2\alpha}},\quad t>0.
\end{gather}
According to \cite[Theorem  3.3.15]{W05}, the ultra-contractility \eqref{ULT-1} of $\widehat{P}_t^\alpha$ is equivalent to the super Poincar\'{e} inequality
\begin{gather*}
	\mu(f^2)
	\leq r\widehat{\Escr}^\alpha(f,f)+c_1(1+r^{-\frac d {2\alpha}})\mu(|f|^2),
	\quad r>0,\;f\in\Dscr(\widehat{\Escr^\alpha})
\end{gather*}
for some constant $c_1>0$. Since $\widehat{\Escr}^\alpha(f,f)=\Escr^\alpha(f,f)$, it follows from \cite[Theorem 3.3.14]{W05} that
\begin{gather*}
	\|P_t^\alpha - \mu\|_{1\to\infty}
	\lesssim (1\wedge t)^{-\frac d {2\alpha}},\quad t>0.
\end{gather*}
Combining this with the definition of $\|P_t^\alpha - \mu\|_{ p\to q}$ for any $t>0$ and $1\leq p\leq q\le\infty$, we see that
\begin{gather*}
	\max\left\{\|P_1^\alpha-\mu\|_{ 1\to 2},\: \|P_1^\alpha - \mu\|_{ 2\to \infty}\right\}
	\leq \|P_1^\alpha -\mu\|_{1\to\infty}<\infty.
\end{gather*}
Together with the Poincar\'e inequality \eqref{Pin} and the semigroup property we conclude that for sufficiently large values of $t>0$
\begin{align*}
	\sup_{x, y \in M} |p_t^\alpha(x, y) - 1|
	&= \|P_t^\alpha-\mu\|_{L^1(\mu)\to L^\infty(\mu)} \\
	&\leq \|P_1^\alpha-\mu\|_{L^1(\mu)\to L^2(\mu)} \|P_{t-2}^\alpha-\mu\|_{L^2(\mu)\to L^2(\mu)}
	\|P_1^\alpha - \mu\|_{L^2(\mu)\to L^\infty(\mu)}\\
	&\leq c_2\eup^{-\lambda_1^\alpha t},
\end{align*}
where $p_t^\alpha$ is the heat kernel of the semigroup $(P_t^\alpha)_{t\geq 0}$. Therefore, we can find some  $s>0$ and $0<c<1$ such that
\begin{gather}\label{h-l}
	\inf_{x, y \in M} p_s^\alpha(x, y)
	\geq 1 - \sup_{x, y \in M} |p_s^\alpha (x, y) - 1|
	\geq c.
\end{gather}
For $s>0$, define
\begin{gather}\label{def of mu}
	\widetilde \mu_{T,s}^\alpha := \frac 1 T\int_0^{T}\delta_{\Xcal_{t+s}^\alpha} \,\dd t.
\end{gather}
Because of \eqref{h-l} we can find for any $\nu \in \Pscr$ some $s>0$ such that
\begin{gather} \label{CC0'}
	\nu_s:= \nu P_s^{\alpha,*} \geq c\mu.
\end{gather}
Combining this with the Markov property shows the following estimate for $\widetilde\mu_{T,s}^\alpha$:
\begin{gather}\label{AB1}
	\inf_{\nu\in \Pscr} \Ee^{\nu}\left[\Ww_1^q(\widetilde\mu_{T,s}^\alpha,\mu)\right]
	= \inf_{\nu\in \Pscr} \Ee^{\nu_s}\left[\Ww_1^q(\mu_T^\alpha,\mu)\right]
	\gtrsim \Ee^\mu\left[\Ww_1^q(\mu_T^\alpha,\mu)\right]
	\gtrsim \gamma_{\alpha,d}(T)^q
\end{gather}
for large $T>0$. On the other hand, it is easy to see that for $T>s$,
\begin{gather*}
	\Pi :=\frac{1}{T}\int_{0}^{s}\left(\delta_{\Xcal_t^\alpha}\otimes \delta_{\Xcal_{t+T}^\alpha} \right) \dd t
	+ \frac{1}{T}\int_{s}^T\left(\delta_{\Xcal_t^\alpha}\otimes \delta_{\Xcal_{t}^\alpha} \right) \dd t
	\in\Cscr(\mu_T^\alpha, \widetilde\mu_{T,s}^\alpha).
\end{gather*}
Since the space $M$ is compact, its diameter $D = \sup_{x,y}\rho(x,y)$ is finite and we see that for all $s<T$
\begin{align*}\label{DPS}
	\Ww_1(\mu_T^\alpha,\widetilde\mu_{T,s}^\alpha)
	&\leq \iint_{M\times M}\rho(x,y) \, \Pi(\dd x,\dd y)\\
	&=   \frac 1T \int_0^s\iint_{M\times M}\rho(x,y) \rho(\Xcal_t^\alpha,\Xcal_{t+T}^\alpha)\, \Pi(\dd x,\dd y)
	\leq \frac{sD}{T}.
\end{align*}
Combining this with Lemma \ref{mu-L}, \eqref{AB1}, \eqref{DPS} and using once again the triangle inequality  (resp.\ quasi triangle inequality if $q\in (0,1)$)  for the Wasserstein distance, we conclude that for large enough $T>0$,
\begin{align*}
 	\inf_{\nu\in \Pscr} \Ee^\nu\left[\Ww_1^q(\mu_{T}^\alpha,\mu)\right]
 	&\geq  \inf_{\nu\in \Pscr} \Ee^\nu\left[2^{-(1\vee q)} \Ww_1^q(\widetilde\mu_{T,s}^\alpha,\mu) -\Ww_1^q(\mu_T^\alpha,\widetilde\mu_{T,s}^\alpha)\right] \\
	&\gtrsim \gamma_{\alpha,d}(T)^q -  T^{-q}
	\gtrsim \gamma_{\alpha,d}(T)^q.
\end{align*}
This completes the proof of Theorem~\ref{Wpq-L}.
\end{proof}

\section{The renormalization limit -- proof of Theorem~\ref{Limit-R}}
In this section, we focus on the case $(\alpha,d)=(1/2,3)$ and $Z=0$. The proof of Theorem~\ref{Limit-R}, originally developed in \cite{AG2019} for the renormalization limit in the i.i.d. setting, was later adapted by \cite{TWZ} to diffusion processes, i.e., $(\alpha,d)=(1,4)$. Here, we extend this framework to subordinated processes. For any $T,\epsilon,\xi>0$, we consider the set
\begin{gather*}
	A_{T,\epsilon}^\xi
	:=\left\{ \omega\in\Omega \::\:  \|\nabla^2 g_{T,\epsilon}^\alpha(\omega)\|_\infty\leq \xi\right\},
\end{gather*}
where $g_{T,\epsilon}^\alpha$ is the function defined in \eqref{g-te}.
\begin{lemma}\label{A-T}
There exists a constant $\gamma>0$, such that for any $T,\epsilon>0$,
\begin{gather*}
	\Pp^\mu\left((A_{T,\epsilon}^\xi)^c\right)
	= \Pp^\mu\left(\|\nabla^2 g_{T,\epsilon}^\alpha\|_\infty > \xi\right)
	\lesssim \epsilon^{-6}\xi^{-4}\exp\left(-\frac{T\xi^2}{2\epsilon^{-1}+\gamma\normal\epsilon^{-1}\xi}\right).
\end{gather*}
\end{lemma}

\begin{proof}
Recall that
\begin{gather*}
	g_{T,\epsilon}^\alpha
	=(-L)^{-1}(f_{T,\epsilon}^\alpha-1)
	=\int_0^\infty P_t (f_{T,\epsilon}^\alpha-1)\,\dd t
	=\frac 1 T \int_0^T q_\epsilon(\Xcal_t^\alpha,\cdot)\,\dd t,
\end{gather*}
where
\begin{gather*}
	q_\epsilon(x,y) := \int_0^\infty[p_{t+\epsilon}(x,y)-1]\,\dd t.
\end{gather*}
 From Propostion~\ref{bern}  we know that
\begin{gather}\label{p-tail}
	\Pp^\mu\left(\left|\nabla_y^2  g_{T,\epsilon}^\alpha(y)\right|>\xi\right)
	\lesssim \exp\left[-\frac{T\xi^2}{2\int_M|\nabla_x L_x^{-3/4}\nabla_y^2 q_\epsilon(x,y)|^2\mu(\dd x)+\gamma\|\nabla_y^2 q_\epsilon(\cdot,y)\|_{L^3(\mu)}\xi}\right].
\end{gather}
Because of the definition of $q_\epsilon$, the integration-by-parts formula \eqref{eq:ibp}, and the $L^2$-boundedness of the Riesz transform, there is a constant $c>0$ such that
\begin{gather*}
	\mathrm{I} := \|\nabla(-L_y)^{-\frac 3 4}\nabla_y^2 q_\epsilon(\cdot,y)\|_{L^2(\mu)}
	\leq c\|\nabla_y^2(-L_y)^{-\frac 1 4}q_\epsilon(\cdot,y)\|_{L^2(\mu)}.
\end{gather*}
Observe that
\begin{align*}
	(-L_y)^{-\frac 1 4}q_\epsilon(x,y)
		&=\frac 1 {\Gamma(1/4)}\int_0^\infty s^{-\frac 3 4} P_tq_\epsilon(x,\cdot)(y)\,\dd t\\
		&=\frac 1 {\Gamma(1/4)}\int_0^\infty\,\dd t\int_0^\infty s^{-\frac 3 4}(p_{t+s+\epsilon}(x,y)-1) \,\dd s,
\end{align*}
and for any $1<p<\infty$,
\begin{gather}\label{sup-p}
	\sup_{y\in M}\|\nabla_y^2(p_t(\cdot,y)-1)\|_{L^p(\mu)}\lesssim t^{-\frac 5 2+\frac 3 {2p}}\eup^{-\lambda_1 t},
\end{gather}
see, e.g., \cite[proof of Lemma 2.4, (2.21)]{TWZ}. Therefore, we have
\begin{align*}
	\mathrm{I}
	&\lesssim \int_0^\infty \,\dd t \int_0^\infty\|\nabla_y^2(p_{t+s+\epsilon}(\cdot,y)-1)\|_{L^2(\mu)} s^{-\frac 3 4}\,\dd s\\
	&\lesssim \int_0^\infty s^{-\frac 3 4}\,\dd s\int_0^\infty \eup^{-\lambda(t+s+\epsilon)}(t+s+\epsilon)^{-\frac{7} 4}\,\dd t\\
	&\lesssim \int_0^\infty  s^{-\frac 3 4}\eup^{-\lambda_1 s}(s + \epsilon)^{-\frac 3 4}\,\dd s
	\lesssim \epsilon^{-\frac 1 2}.
\end{align*}
This implies that
\begin{gather}\label{qp}
	\sup_{y \in M} \int_M \left|\nabla_x L_x^{-\frac 3 4} \nabla_y^2 q_\epsilon(x,y)\right|^2 \,\mu(\dd x)
	\lesssim \epsilon^{-1}.
\end{gather}
On the other hand, we can use \eqref{sup-p} to see that
\begin{gather*}
	\sup_{y\in M}\|\nabla_y^2 q_\epsilon(\cdot,y)\|_{L^3(\mu)}
	\leq \int_\epsilon^\infty\sup_{y\in M}\|\nabla_y^2(p_t(\cdot,y)-1)\|_{L^3(\mu)}\,\dd t
	\lesssim \int_\epsilon^\infty t^{-2}\eup^{-\lambda_1 t}\,\dd t
	\lesssim \epsilon^{-1}.
\end{gather*}
Combining this with \eqref{qp} and \eqref{p-tail}, we arrive at
\begin{gather*}
	\Pp^\mu\left(|\nabla_y^2 g_{T,\epsilon}^\alpha(y)|>\xi\right)
	\lesssim\exp\left[-\frac{T\xi^2}{2\epsilon^{-1}+\gamma\epsilon^{-1}\xi}\right].
\end{gather*}
Furthermore, by \cite[Remark 2.5]{TWZ},
\begin{gather*}
	K := \|\nabla^3 g_{T,\epsilon}^\alpha\|_\infty
	\lesssim \epsilon^{-\frac 3 2}.
\end{gather*}
In order to deal with the supremum inside the probability measure, we take a suitable net with cell diameter $l = \xi/K$, hence with $N(l)\lesssim (K/\xi)^4$ cells. This gives
\begin{gather*}
	\Pp^\mu\left(\sup_{y\in M}\left|\nabla_y^2  g_{T,\epsilon}^\alpha(y)\right|> \xi\right)
	\lesssim N(l)\exp\left[-\frac{T\xi^2}{2\epsilon^{-1}+\gamma\epsilon^{-1}\xi}\right]
	\lesssim \epsilon^{-6}\xi^{-4}\exp\left[-\frac{T\xi^2}{2\epsilon^{-1}+\gamma\epsilon^{-1}\xi}\right].
\end{gather*}
This finishes the proof.
\end{proof}

Assume that $T>1$ is large and choose in Lemma~\ref{A-T} $\epsilon=\log^\gamma T /T$ with a suitable $\gamma$. Then we have:
\begin{lemma}\label{Limit-M}
	For some $\gamma>3$ it holds that
	\begin{gather*}
		\left|\frac T {\log T }\Ee^\mu\left[\Ww_2^2(\mu_{T,\epsilon}^\alpha,\mu)\right]-\frac {\mathrm{Vol}(M)} {2\pi^2}\right|
		\lesssim \frac 1{\log T}, \quad T\gg 1.
	\end{gather*}
\end{lemma}
\begin{proof}
We use Lemma~\ref{A-T} with $\xi=1/\log T$ and $\epsilon=\log^\gamma T /T$ for large enough $T>1$.  This gives
\begin{gather*}
	\Pp^\mu\left((A_{T,\epsilon}^\xi)^c\right)
	\lesssim \exp\left[- c\log^{\gamma-2}T\right], \quad T\gg 1
\end{gather*}
for some constant $c>0$. We denote by $\hat{\mu}_{T,\epsilon}^\alpha := {\exp\left(\nabla  g_{T,\epsilon}^\alpha \right)_{\#}\mu}$ the push-forward (image measure) of $\mu$ under the mapping $\exp(\nabla  g_{T,\epsilon}^\alpha\normal )$, i.e.\
\begin{gather*}
	\exp(\nabla  g_{T,\epsilon}^\alpha\normal )_{\#}\mu (B)
	:= \mu\left((\exp(\nabla  g_{T,\epsilon}^\alpha\normal ))^{-1}(B)\right)
\end{gather*}
for all Borel sets $B\subset M$. From \cite[Corollary 3.3]{G2019} or \cite[(6.3)]{AG2019} we know that
\begin{gather*}
	\Ww_2\left(\hat{\mu}_{T,\epsilon}^\alpha,\mu\right)
	= \|\nabla  g_{T,\epsilon}^\alpha\normal \|_{L^2(\mu)}.
\end{gather*}
Moreover, the following inequality holds on the set $A_{T,\epsilon}^\xi$, see \cite[(6.2)]{AG2019}:
\begin{gather*}
	\Ww_2^2\left(\mu_{T,\epsilon}^\alpha,\hat{\mu}_{T,\epsilon}^\alpha\right)
	\lesssim \xi^2\mu(|\nabla  g_{T,\epsilon}^\alpha\normal |^2).
\end{gather*}
The lower triangle inequality now gives
\begin{gather*}
	\left|\Ww_2\left(\mu_{T,\epsilon}^\alpha,\mu\right) - \Ww_2\left(\hat{\mu}_{T,\epsilon}^\alpha,\mu\right)\right|
	\leq \Ww_2\left(\mu_{T,\epsilon}^\alpha,\hat{\mu}_{T,\epsilon}^\alpha\right)
\end{gather*}
and this leads to
\begin{gather}\begin{aligned}\label{W2-s}
	\Ee^\mu&\left[\left|\Ww_2\left(\mu_{T,\epsilon}^\alpha,\mu\right)
	-\|\nabla  g_{T,\epsilon}^\alpha\normal \|_{L^2(\mu)}\right|^2\right]\\
	&=\Ee^\mu\left[\left|\Ww_2\left(\mu_{T,\epsilon}^\alpha,\mu\right)
	-\Ww_2\left(\hat{\mu}_{T,\epsilon}^\alpha,\mu\right)\right|^2\right]\\
	&\lesssim \Pp^\mu\left((A_{T,\epsilon}^\xi)^c\right) + \Ee^\mu\left[\I_{A_{T,\epsilon}^{\xi}}\Ww_2^2(\mu_{T,\epsilon}^\alpha,\hat{\mu}_{T,\epsilon}^\alpha)\right]\\
	&\lesssim \Pp^\mu\left((A_{T,\epsilon}^\xi)^c\right) + \Ee^\mu\left[\I_{A_{T,\epsilon}^{\xi}}\xi^2\mu(|\nabla  g_{T,\epsilon}^\alpha\normal |^2)\right]\\
	&\lesssim \exp\left[- c\log^{\gamma-2}T\right] + \xi^2\Ee^\mu\left[\mu\left(|\nabla  g_{T,\epsilon}^\alpha\normal |^2\right)\right].
\end{aligned}\end{gather}
We will now find an upper bound for the term $\Ee[\mu(|\nabla  g_{T,\epsilon}^\alpha\normal |^2)]$. In dimension $d=3$, \cite[Corollary 3.2]{CR} gives the following short-time asymptotic behaviour of the heat trace:
\begin{gather}\label{trace}
	\sum_{i=1}^\infty \eup^{-\lambda_i \epsilon}
	= \mathrm{tr}\,\eup^{t L}-1
	= \frac{\mathrm{Vol}(M)}{(4\pi \epsilon)^{\frac 3 2}} + O(\epsilon^{-\frac 1 2}),
	\quad \epsilon\to 0.
\end{gather}
From \eqref{lam} we take that $\lambda_i\sim i^{\frac 2 3}$ if $d=3$. Thus,
\begin{gather*}
	\sum_{i=1}^\infty\int_1^\infty \frac 1 {\lambda_i}\eup^{-(2\epsilon+s)\lambda_i}s^{-\frac 1 2}\,\dd s
	\leq \sum_{i=1}^\infty\int_1^\infty \frac 1 {\lambda_i}\eup^{-\lambda_i s}\,\dd s
	= \sum_{i=1}^\infty\frac {1}{\lambda_i^2}\eup^{-\lambda_i}
	\leq \sum_{i=1}^\infty\frac 1 {i^{\frac 4 3}}
	<\infty.
\end{gather*}
This, together with \eqref{trace} and with a few elementary calculations shows
\begin{align*}
	\sum_{i=1}^\infty\frac {\eup^{-2\lambda_i\epsilon}}{\lambda_i^{\frac 3 2}}
	&= \sum_{i=1}^\infty\frac 1 {\Gamma(\frac 1 2)}\int_0^\infty\frac 1 {\lambda_i}\eup^{-(2\epsilon+s)\lambda_i}s^{-\frac 1 2}\,\dd s\\
	&= \sum_{i=1}^\infty\frac 1 {\sqrt{\pi}}\int_0^1\frac 1 {\lambda_i}\eup^{-(2\epsilon+s)\lambda_i}s^{-\frac 1 2}\,\dd s+O(1)\\
	&= \frac 1 {\sqrt{\pi}}\int_0^1\int_{2\epsilon+s}^\infty s^{-\frac 1 2} \sum_{i=1}^\infty\eup^{-\lambda_i u}\,\dd u \,\dd s+O(1)\\
	&= \frac 1 {\sqrt{\pi}}\int_0^1\int_{2\epsilon+s}^2 s^{-\frac 1 2} \sum_{i=1}^\infty\eup^{-\lambda_i u}\,\dd u \,\dd s+\frac 1 {\sqrt{\pi}}\int_0^1 \int_2^\infty s^{-\frac 1 2} \sum_{i=1}^\infty\eup^{-\lambda_i u}\,\dd u \,\dd s+O(1)\\
	&= \frac 1 {\sqrt{\pi}}\int_0^1\int_{2\epsilon+s}^2\frac{\mathrm{Vol}(M)}{(4\pi u)^{\frac 3 2}}+O(u^{-\frac 1 2})\,\dd u s^{-\frac 1 2}\,\dd s+O(1)\\
	&= \frac{\mathrm{Vol}(M)} {4\pi^2}\int_0^1(2\epsilon+s)^{-\frac 1 2}s^{-\frac 1 2}\,\dd s+O(1)\\
	&= \frac{\mathrm{Vol}(M)} {2\pi^2}\log\left(\sqrt{1+\frac 1 {2\epsilon}}+\sqrt{\frac 1 {2\epsilon}}\right)+O(1),\quad \epsilon\to 0.
\end{align*}
In the last equality we use the following well-known primitive
\begin{gather*}
	\int \frac {\dd x}{\sqrt{x^2+2\epsilon x}} = \log\left(\epsilon + x + \sqrt{x^2+2\epsilon x}\right) + C.
\end{gather*}
Recall that $\epsilon=\log^\gamma T\normal /T$. Using \eqref{E-g} and $(\alpha,d)=(1/2,3)$ in \eqref{phiL-0}, we get
\begin{gather}\begin{aligned}\label{Limit}
	&\left|\frac T {\log T}\Ee^\mu\left[\|\nabla  g_{T,\epsilon}^\alpha \|_{L^2(\mu)}^2\right]
	-\frac {\mathrm{Vol}(M)} {2\pi^2}\right|\\
	&=\left|\frac 1 {\log T }\sum_{i=1}^\infty \frac{\eup^{-2\lambda_i \epsilon}}{\lambda_i}\Ee^\mu\left[|\psi_i^\alpha(T)|^2\right]
	-\frac {\mathrm{Vol}(M)} {2\pi^2}\right|\\
	&\leq \left|\frac 1 {\log T}\sum_{i=1}^\infty \frac{\eup^{-2\lambda_i \epsilon}}{\lambda_i} \left[\Ee^\mu[|\psi_i^\alpha(T)|^2]-\frac 2 {\sqrt{\lambda_i}}\right]\right|
	+ \left|\frac 2 {\log T }\sum_{i=1}^\infty \frac{\eup^{-2\lambda_i \epsilon}}{\lambda_i^{\frac 3 2}}-\frac {\mathrm{Vol}(M)} {2\pi^2}\right|\\
	&\leq \frac 1 {T\log T}\sum_{i=1}^\infty \frac{\eup^{-2\lambda_i\epsilon}}{\lambda_i^2}
	+ \frac {\mathrm{Vol}(M)} {\pi^2}
	\left|\frac{\log\left[\sqrt{1+\frac 1 {2\epsilon}}+\sqrt{\frac 1 {2\epsilon}}\right]}{\log T }
	- \frac 12\right| + O\left(\frac 1{\log T}\right)\\
	&\lesssim \frac 1 {T\log T } + \frac{\log\log T }{\log T } + O\left(\frac 1{\log T}\right)
	\lesssim \frac{\log\log T }{\log T },\quad T\gg 1.
\end{aligned}\end{gather}
In the calculation above we use our choice of $\epsilon$ along with
\begin{gather*}
	\sum_{i=1}^\infty\frac{\eup^{-2\lambda_i \epsilon}}{\lambda_i^2}
	\lesssim \sum_{i=1}^\infty \frac{1}{i^{\frac 4 3}}
	< \infty,
\intertext{and}
	\left|\log\left[\sqrt{1+\frac T {2\log^\gamma T}} + \sqrt{\frac T {2\log^\gamma T }}\right]
	-\frac 1 2 \log T \right|\\
	\leq \log \sqrt{T} - \log \sqrt{\frac T{2\log^\gamma T}}
	= O(\log \log T ),\quad  T\gg 1.
\end{gather*}
From \eqref{Limit} we also see that
\begin{gather*}
	 \Ee^\mu\left[\|\nabla  g_{T,\epsilon}^\alpha\normal \|_{L^2(\mu)}^2\right] \normal
	= \Ee^\mu\left[\mu\left(|\nabla  g_{T,\epsilon}^\alpha\normal |^2\right)\right]
	= O\left(T^{-1}\log T \right),\quad  T\gg 1.
\end{gather*}
This,  \eqref{W2-s} and our choice $\xi=1/\log T$,   yield
\begin{gather}\label{W2^2}
	\Ee^\mu\left[\left|\Ww_2\left(\mu_{T,\epsilon}^\alpha,\mu\right) - \|\nabla  g_{T,\epsilon}^\alpha\normal \|_{L^2(\mu)}\right|^2\right]
	=
	O\left(\frac 1{T\log T}\right),\quad   T\gg 1.
\end{gather}
Furthermore, the Cauchy-Schwarz inequality shows that
\begin{align*}
	&\Ee^\mu\left[\left|\Ww_2^2\left(\mu_{T,\epsilon}^\alpha,\mu\right)-\|\nabla g_{T,\epsilon}^\alpha\normal \|_{L^2(\mu)}^2\right|\right]\\
	&\quad= \Ee^\mu\left[\left|\left(\Ww_2\left(\mu_{T,\epsilon}^\alpha,\mu\right) -\|\nabla  g_{T,\epsilon}^\alpha\normal \|_{L^2(\mu)}\right)
	\left(\Ww_2\left(\mu_{T,\epsilon}^\alpha,\mu\right) + \|\nabla g_{T,\epsilon}^\alpha\normal \|_{L^2(\mu)}\right)\right|\right]\\
	&\quad\leq \left(\Ee^\mu\left[\left|\Ww_2\left(\mu_{T,\epsilon}^\alpha,\mu\right)-\|\nabla g_{T,\epsilon}^\alpha\normal \|_{L^2(\mu)}\right|^2\right]\right)^{1/2}\times\mbox{}\\
	&\qquad\qquad\mbox{}\times
 	\left(2\Ee^\mu\left[\left|\Ww_2\left(\mu_{T,\epsilon}^\alpha,\mu\right) - \|\nabla  g_{T,\epsilon}^\alpha\normal \|_{L^2(\mu)}\right|^2\right] + 4\Ee^\mu\left[\mu\left(|\nabla g_{T,\epsilon}^\alpha\normal |^2\right)\right]\right)^{1/2}\\
	&\quad\lesssim T^{-\frac 1 2}\log^{-\frac 1 2}T \left(T^{-\frac 1 2}\log^{-\frac 1 2}T+T^{-\frac 1 2}\log^{\frac 1 2}T\right)
	= O\left(T^{-1}\right).
\end{align*}
This and \eqref{Limit} finally prove
\begin{gather*}
	\left|\frac T {\log T }\Ee^\mu\left[\Ww_2^2\left(\mu_{T,\epsilon}^\alpha,\mu\right)\right] - \frac{\mathrm{Vol}(M)} {2\pi^2}\right|
	= O(\log^{-1}(T)).
	\qedhere
\end{gather*}
\end{proof}

Recall our choice $\epsilon=\log^\gamma T \normal/T$.
\begin{lemma}\label{eps}
For some $\gamma>3$ it holds that
\begin{gather*}
	\Ee^\mu\left[\Ww_2^2\left(\mu_T^\alpha,\mu_{T,\epsilon}^\alpha\right)\right]
	\lesssim \frac{\log\log T }{T},\quad T\gg 1.
\end{gather*}
\end{lemma}
\begin{proof}
Pick $\xi=1/\log T$, $T\ge 1$. By  Lemma \ref{A-T}, there exists some constant $c>0$ such that
\begin{gather}\begin{aligned}\label{W2-eps}
	\Ee^\mu\left[\Ww_2^2\left(\mu_T^\alpha,\mu_{T,\epsilon}^\alpha\right)\right]
	&= \Ee^\mu\left[\Ww_2^2\left(\mu_T^\alpha,\mu_{T,\epsilon}^\alpha\right)\I_{A_{T,\epsilon}^\xi}\right]
	+ \Ee^\mu\left[\Ww_2^2\left(\mu_T^\alpha,\mu_{T,\epsilon}^\alpha\right)\I_{(A_{T,\epsilon}^\xi)^c}\right]\\
	&\lesssim \Ee^\mu\left[\Ww_2^2\left(\mu_T^\alpha,\mu_{T,\epsilon}^\alpha\right)\I_{A_{T,\epsilon}^\xi}\right]
	+ \Pp^\mu\left((A_{T,\epsilon}^\xi)^c\right)\\
	&\lesssim \Ee^\mu\left[\Ww_2^2\left(\mu_T^\alpha,\mu_{T,\epsilon}^\alpha\right)\I_{A_{T,\epsilon}^\xi}\right]
	+ \exp\left[-c \log^{\gamma-2}(1+T)\normal\right].
\end{aligned}\end{gather}
It is, therefore, enough to bound $\Ee\left[\Ww_2^2\left(\mu_T^\alpha,\mu_{T,\epsilon}^\alpha\right)\I_{A_{T,\epsilon}^\xi}\right]$. On the event $A_{T,\epsilon}^\xi$, we have
\begin{gather*}
	\|f_{T,\epsilon}^\alpha-1\|_\infty=\|L g_{T,\epsilon}^\alpha\|_\infty\lesssim \|\nabla^2 g_{T,\epsilon}^\alpha\|_\infty
	\leq \frac 1{\log T},\quad T\gg 1.
\end{gather*}
This shows that, for sufficiently large $T$ we have $\inf_{x\in M} f_{T,\epsilon}^\alpha(x)\geq 1/2$. Combining this with \cite[Theorem A.1]{W22} and \eqref{phiL-0}, we get for any $\epsilon'<\epsilon$,
\begin{align*}
	\Ee^\mu\left[\Ww_2^2\left(\mu_{T,\epsilon}^\alpha,\mu_{T,\epsilon'}^\alpha\right)\right]
	&\lesssim \Ee^\mu\left[\int_M \left|\nabla(-L)^{-1}(f_{T,\epsilon}^\alpha-f_{T,\epsilon'}^\alpha)\right|^2\,\dd\mu\right]\\
	&= \frac 1 T \sum_{i=1}^\infty \frac{\left(\eup^{-\lambda_i \epsilon'}-\eup^{-\lambda_i \epsilon}\right)^2} {\lambda_i} \Ee^\mu\left[|\psi_i^\alpha(T)|^2\right]\\
	&\leq \frac 1 T \sum_{i=1}^\infty \frac{\eup^{-2\lambda_i \epsilon'}-\eup^{-2\lambda_i \epsilon}}{\lambda_i} \Ee^\mu\left[|\psi_i^\alpha(T)|^2\right]\\
	&= \int_{\epsilon'}^\epsilon\frac 1 T \sum_{i=1}^\infty \eup^{-2\lambda_i u} \Ee^\mu\left[|\psi_i^\alpha(T)|^2\right] \dd u\\
	&\lesssim \int_{\epsilon'}^\epsilon \frac 1 T \sum_{i=1}^\infty \eup^{-2\lambda_i u} \left(\lambda_i^{-\frac 1 2}+\lambda_i^{-1}T^{-1}\right) \dd u\\
	&\lesssim \int_{\epsilon'}^\epsilon \frac 1 T \sum_{i=1}^\infty \eup^{-2\lambda_i u}\lambda_i^{-\frac 1 2}\,\dd u\\
	&\lesssim \frac 1T\log\left(\frac{\epsilon}{\epsilon'}\right),\quad  T\gg1.
\end{align*}
For the last inequality we observe that
\begin{gather*}
	\sum_{i=1}^\infty \eup^{-2\lambda_i u}\lambda_i^{-\frac 1 2}
	\lesssim \sum_{i=1}^\infty \eup^{-2i^{2/3}u} \left(i^{2/3}\right)^{-\frac 12}
	\lesssim \int_1^\infty \eup^{-2s^{2/3}u} \left(s^{2/3}\right)^{-\frac 12}\,\dd s
	\lesssim \frac 1u^,\quad u>0.
\end{gather*}
For large enough $T>1$ we may now take $\epsilon'=\log\log T\normal /T$ and combine the above inequality with \eqref{W2-eps}. This finally gives
\begin{gather*}
	\Ee^\mu\left[\Ww_2^2\left(\mu_T^\alpha,\mu_{T,\epsilon}^\alpha\right)\right]
	\lesssim \frac{\log\log T }{T} + \exp\left[-c \log^{\gamma-2}T\normal\right]
	\lesssim \frac{\log\log T\normal }{T},\quad  T\gg 1.
\qedhere
\end{gather*}
\end{proof}

\begin{proof}[Proof of Theorem \ref{Limit-R}]
Combining Lemma \ref{eps} and Theorem \ref{Wpq} for $(\alpha,d)=(\frac 12,3)$, gives
\begin{gather*}\begin{aligned}
	&\left|\Ee^\mu\left[\Ww_2^2\left(\mu_T^\alpha,\mu\right)\right]
	- \Ee^\mu\left[\Ww_2^2\left(\mu_{T,\epsilon}^\alpha,\mu\right)\right]\right|\\
	&\leq \Ee^\mu\left[\Ww_2^2\left(\mu_T^\alpha,\mu_{T,\epsilon}^\alpha\right)\right] + 2\sqrt{\Ee^\mu\left[\Ww_2^2\left(\mu_T^\alpha,\mu_{T,\epsilon}^\alpha\right)\right]
		\Ee^\mu\left[\Ww_2^2\left(\mu_T^\alpha,\mu\right)\right]}\\
	&\lesssim \frac{\sqrt{\log\log T\normal \log T\normal }}{T},\quad  T\gg 1.
\end{aligned}\end{gather*}
This, together with Lemma \ref{Limit-M} shows that
\begin{gather}\begin{aligned}\label{W2-L}
	&\left|\frac T {\log T\normal }\Ee^\mu\left[\Ww_2^2\left(\mu_T^\alpha,\mu\right)\right] - \frac {\mathrm{Vol}(M)} {2\pi^2}\right|\\
	&\quad\leq \left|\frac T {\log(1+T)\normal } \Ee^\mu\left[\Ww_2^2\left(\mu_T^\alpha,\mu\right)\right]
	- \frac T {\log T\normal }\Ee^\mu\left[\Ww_2^2\left(\mu_{T,\epsilon}^\alpha,\mu\right)\right]\right|\\
	&\quad\qquad\mbox{}+ \left|\frac T {\log T\normal } \Ee^\mu\left[\Ww_2^2\left(\mu_{T,\epsilon}^\alpha,\mu\right)\right]
	- \frac {\mathrm{Vol}(M)} {2\pi^2}\right|\\
	&\quad\lesssim \frac T {\log T } \frac{\sqrt{\log\log T \log T }}{T}
	+ \frac 1 {\log T }\\
	&\quad\lesssim \sqrt{\frac{\log\log T }{\log T }},\quad T\gg1.
\end{aligned}\end{gather}
By \cite[Page 11]{TWZ}, there exists some constant $c>0$ such that
\begin{gather*}
	\Ee\left[\rho^2(X_t^x,X_t^\mu)\right]
	\lesssim \eup^{-c t},\quad t>0,
\end{gather*}
where $(X_t^x,X_t^\mu)$ denotes a coupling of the process $X_t$ with initial distributions $\delta_x$ and $\mu$ respectively. Note that we use for couplings $\Pp$ and $\Ee$ without superscript, since the starting law is indicated in the random variable.
Given $\alpha = \frac 12$, $Z=0$ and $\Ee\left[\eup^{-c S_t^{\alpha}}\right] = \eup^{-\sqrt{c}t}$, it follows  from the subordination representation $\Xcal_t^{\alpha} = X_{S_t^\alpha}$  that
\begin{gather*}
	\Ee\left[\rho^2(\Xcal_t^{\alpha,x},\Xcal_t^{\alpha,\mu})\right]
	\lesssim \eup^{-\sqrt{c} t},\quad t>0.
\end{gather*}
We denote by $\mu_T^x$ and $\mu_T^\mu$ the empirical measure of $\Xcal_t^{\alpha,x}$ and $\Xcal_t^{\alpha,\mu}$ respectively. Consequently,
\begin{gather*}\begin{aligned}
	\Ee\left[\Ww_2^2\left(\mu_T^{\alpha,x},\mu_T^{\alpha,\mu}\right)\right]
	\leq \frac 1 T \int_{0}^T \Ee\left[\rho^2(\Xcal_t^{\alpha,x},\Xcal_t^{\alpha,\mu})\right] \dd t
	\lesssim \frac 1 T, \quad T\gg 1.
\end{aligned}\end{gather*}
Combining this with Theorem \ref{Wpq}, yields
\begin{gather*}
	\Ee^x\left[\Ww_2\left(\mu_T^\alpha,\mu\right)\right]
	\lesssim \Ee\left[\Ww_2^2\left(\mu_T^{\alpha,x},\mu_T^{\alpha,\mu}\right)\right]
	+ \Ee^\mu\left[\Ww_2^2\left(\mu_T^\alpha,\mu\right)\right]
	\lesssim \frac{\log T } T,\quad T\gg 1.
\end{gather*}
The above estimate,  the triangle inequality for $\Ww_2$ and Theorem \ref{Wpq} give
\begin{align*}
	&\frac T {\log T  }\left|\Ee^x\left[\Ww_2^2\left(\mu_T^\alpha,\mu\right)\right]
	- \Ee^{\mu}\left[\Ww_2^2\left(\mu_T^\alpha,\mu\right)\right]\right|\\
	&\qquad\lesssim \frac T {\log T } \left(\Ee\left[\Ww_2^2\left(\mu_T^{\alpha,x},\mu_T^{\alpha,\mu}\right)\right]\right)^{\frac 1 2} \left(\Ee^x\left[\Ww_2^2\left(\mu_T^\alpha,\mu\right)\right]\right)^{\frac 1 2}
	\left(\Ee^\mu\left[\Ww_2^2\left(\mu_T^\alpha,\mu\right)\right]\right)^{\frac 1 2}\\
	&\qquad\lesssim \frac 1 {\sqrt{\log T }},\quad T\gg 1.
\end{align*}
Finally, this and \eqref{W2-L}, prove the claim of Theorem \ref{Limit-R}.
\end{proof}

\paragraph{Acknowledgement.}
R.L.\ Schilling is supported by the 6G-Life project (BMFTR Grant No.~16KISK001K) and ScaDS.AI (TU Dresden, Univ.~Leipzig)
centre. B.\ Wu is supported  by the National Key R\& D Program of China (Grant No. 2023YFA1010400, 2022YFA1006003), the National Natural Science Foundation of China
 (Grant No. 12401174), the  Natural Science Foundation-Fujian (Grant No. 2024J08051), Fujian Alliance of Mathematics (Grant No. 2023SXLMQN02), the Education and Scientific Research Project for Young and Middle-aged Teachers in Fujian
Province of China (Grant No. JAT231014) and the Alexander von Humboldt Foundation.

\end{document}